\documentclass[hidelinks,onefignum,onetabnum]{siamart251216}

\usepackage{lipsum}
\usepackage{amsfonts}
\usepackage{graphicx}
\usepackage{epstopdf}
\usepackage{algorithmic}
\ifpdf
  \DeclareGraphicsExtensions{.eps,.pdf,.png,.jpg}
\else
  \DeclareGraphicsExtensions{.eps}
\fi

\newsiamremark{remark}{Remark}
\newsiamremark{hypothesis}{Hypothesis}
\crefname{hypothesis}{Hypothesis}{Hypotheses}
\newsiamthm{claim}{Claim}
\newsiamremark{fact}{Fact}
\crefname{fact}{Fact}{Facts}

\theoremstyle{plain}
\theoremheaderfont{\normalfont\sc}
\theorembodyfont{\normalfont\itshape}
\theoremseparator{.}
\theoremsymbol{}
\newtheorem{assumption}{Assumption}[section]

\headers{MultiLevel Gradient-Regularized Newton}{Nick Tsipinakis and Matthias Voigt}

\title{Global $\mathcal{O}\big(\lowercase{k}^{-2}\big)$ Convergence of Multilevel Gradient-Regularized Newton Methods
}

\author{
Nick Tsipinakis%
\thanks{Faculty of Mathematics and Computer Science, UniDistance Suisse, Brig, Switzerland
  (\email{nikolaos.tsipinakis@unidistance.ch},
   \email{matthias.voigt@fernuni.ch}).}
\and
Matthias Voigt\footnotemark[1]
}

\usepackage{amsopn}

\newcommand{\R}{\mathbb{R}}
\newcommand{\N}{\mathbb{N}}
\usepackage{mathtools}
\usepackage{amssymb}
\newcommand{\C}{\mathcal{C}}
\newcommand{\F}{\mathcal{F}}
\newcommand{\I}{\mathcal{I}}
\newcommand{\J}{\mathcal{J}}
\newcommand{\Op}{\mathcal{O}}
\newcommand{\norm}[1]{\left\lVert #1 \right\rVert}
\newcommand{\Hk}{\nabla^2 f_k}
\newcommand{\grad}{\nabla}
\newcommand{\hess}{\nabla^2}
\newcommand{\gk}{\grad f_k}
\newcommand{\gkl}{\grad f_{k_\ell}}
\newcommand{\gkp}{\grad f_{k+1}}

\usepackage{enumitem}

\ifpdf
\hypersetup{
  pdftitle={Global $\mathcal{O}(\lowercase{k}^{-2})$ Convergence of Multilevel Gradient-Regularized Newton Methods},
  pdfauthor={N. Tsipinakis and M. Voigt}
}
\fi

\usepackage[numbers,sort&compress]{natbib}
\usepackage{bm}

\usepackage{pgfplots}
\pgfplotsset{compat=1.18}
\definecolor{mlrncolor}{rgb}{0.0,0.549,0.620}
\definecolor{nestedcolor}{rgb}{0.690,0.188,0.565}
\definecolor{rncolor}{gray}{0.067}
\definecolor{tlrncolor}{rgb}{0.133,0.369,0.659}
\pgfplotsset{
  mlrnaxis/.style={
    scale only axis,
    width=0.37\textwidth, height=0.2\textwidth,
    grid=major,
    major grid style={gray!25, line width=0.3pt},
    tick label style={font=\footnotesize},
    label style={font=\footnotesize},
    xlabel style={yshift=2pt},
    ylabel style={yshift=-2pt},
    tick align=outside, tick pos=left,
    major tick length=2.5pt,
    axis line style={line width=0.5pt},
    scaled x ticks=false,
    xticklabel style={/pgf/number format/fixed},
    max space between ticks=40,
    clip mode=individual,
    legend columns=4,
    legend style={font=\footnotesize, draw=gray!50, line width=0.4pt,
      /tikz/every even column/.append style={column sep=1.2em}},
    legend image code/.code={\draw[##1] (0cm,0cm) -- (0.75cm,0cm);},
  },
  mlrn/.style={color=mlrncolor, line width=1.1pt, solid},
  mlrnnested/.style={color=nestedcolor, line width=1.1pt, dash pattern=on 4pt off 2pt},
  rn/.style={color=rncolor, line width=1.1pt, dash pattern=on 5pt off 1.5pt on 1.1pt off 1.5pt},
  tlrn/.style={color=tlrncolor, line width=1.3pt, dash pattern=on 0.1pt off 2.2pt, line cap=round},
}
\DeclareMathAlphabet{\mathbfcal}{OMS}{cmsy}{b}{n}

\usepackage{algorithm}

\newcommand{\dd}{\mathrm{d}}

\begin{document}

\maketitle


\begin{abstract}
We develop and analyze multilevel gradient-regularized Newton methods for unconstrained convex optimization with Lipschitz-continuous Hessians. By computing Newton corrections in lower-dimensional subspaces, the proposed methods aim to reduce computational costs while preserving the global $\Op\big(k^{-2}\big)$ convergence rate in function values of the full-space method. To establish this rate for multilevel methods, we introduce a class of row-nested restriction operators that ensures monotonicity of the reduced-gradient norms across levels. Numerical experiments with row-nested operators demonstrate convergence behavior comparable to full-space gradient-regularized Newton, while offering substantial computational savings on large-scale instances.
\end{abstract}

\begin{keywords}
multilevel optimization, regularized Newton methods, convex optimization, row-nested restriction operators
\end{keywords}

\begin{MSCcodes}
90C25, 90C30, 65K05, 90C06
\end{MSCcodes}

\section{Introduction}
Cubic regularization of Newton's method \cite{Griewank1981Modification,Nesterov2006Cubic} is a fundamental second-order approach for smooth, possibly nonconvex optimization, achieving an $\Op(\varepsilon^{-3/2})$ complexity for first-order stationarity in nonconvex problems and a global $\Op\big(k^{-2}\big)$ function-value rate for convex objectives with Lipschitz-continuous Hessians. Subsequent developments include accelerated cubic Newton methods attaining an $\Op(k^{-3})$ rate \cite{nesterov2008accelerating} and Adaptive Regularization with Cubics (ARC) \cite{cartis2011adaptive,Cartis2011AdaptiveII,Cartis2012Evaluation}. However, cubic regularization requires solving a potentially expensive subproblem at each iteration.

Gradient-regularized Newton methods provide a simpler alternative, with iterations of the form
\begin{equation} \label{eq:gradient-regularized-newton}
    x_{k+1}=x_k-\left(\nabla^2 f(x_k)+a_k I\right)^{-1}\nabla f(x_k),
\end{equation}
where $a_k>0$. Following the early work of Polyak \cite{polyak2009regularized}, Mishchenko \cite{mishchenko2023regularized} and Doikov and Nesterov \cite{doikov2024gradient} established that choosing $a_k$ proportional to $\sqrt{L\|\nabla f(x_k)\|}$ yields the global $\Op\big(k^{-2}\big)$ function-value rate for convex objectives with Lipschitz-continuous Hessians, improving upon the $\Op(k^{-1})$ rate of standard gradient descent. Nevertheless, computing full-dimensional Hessians and solving the corresponding linear systems can be prohibitively expensive for large-scale problems.

Inspired by classical multigrid methods \cite{briggs2000multigrid}, multilevel optimization reduces these costs by computing corrections in lower-dimensional spaces, replacing full-dimensional Newton systems with smaller ones. Multilevel strategies have been successfully developed for Newton and regularized Newton methods \cite{gratton2008recursive, tsipinakis2023lowrank, ho2019newton, tsipinakis2024multilevel}, higher-order regularization methods, including ARC \cite{calandra2021high}, as well as line-search, preconditioned, accelerated gradient, and proximal methods \cite{WenGoldfarb2010, tsipinakis2023simba, Hovhannisyan2016MAGMA, ang2024mgprox}.

In a foundational work, Gratton, Sartenaer, and Toint \cite{gratton2008recursive} introduced recursive multilevel trust-region (RMTR) methods, establishing an $\Op\big(\varepsilon^{-2}\big)$ worst-case complexity bound for first-order stationarity in possibly nonconvex problems. Their framework permits recursion to a coarser level when
\begin{equation}\label{ineq:rmtr-conditions}
    \norm{R_i g_{i,k}} \geq \kappa_g \norm{g_{i,k}},
    \qquad
    \norm{R_i g_{i,k}} > \varepsilon^g_{i-1},
\end{equation}
where $g_{i,k}$ is the gradient of the current level-dependent model, $\kappa_g>0$ is a prescribed parameter, and $R_i \in \R^{n_{i-1}\times n_i}$ restricts information to the next coarser level. Importantly, this eligibility condition is evaluated before computing a coarse correction, avoiding additional trial steps for level selection.

In general, second- or higher-order multilevel methods need not inherit the fast convergence rates of their full-space counterparts, since coarse corrections exploit only partial second- or higher-order information. Notable results preserving full-space convergence guarantees include RMTR \cite{gratton2008recursive} and the high-order multilevel methods of Calandra et al. \cite{calandra2021high}, which attain the $\Op\big(\varepsilon^{-3/2}\big)$ first-order stationarity complexity of cubic regularization for possibly nonconvex problems. More recently, for convex objectives, the multilevel gradient-regularized Newton method \cite{tsipinakis2025regularized} was shown to achieve the global $\Op\big(k^{-2}\big)$ function-value rate, while the adaptive multilevel Newton method \cite{Tsipinakis2025Adaptive} achieves local quadratic convergence for strongly convex objectives.

However, the deterministic formulations of these last two methods rely on a stronger level-selection criterion, namely
\begin{equation} \label{ineq:condtion-trial-step}
    \norm{R_i \grad f(x_k)} \geq \sigma \norm{\grad f(x_k)},
    \qquad
    \norm{R_i \grad f(x_{k+1})} \geq \sigma \norm{\grad f(x_{k+1})}.
\end{equation}
The second condition requires computing a trial coarse step to evaluate the gradient at the prospective next iterate, potentially increasing the computational cost. In contrast, the classical eligibility criterion \eqref{ineq:rmtr-conditions} uses only current-iterate information. To the best of our knowledge, a direct global $\Op\big(k^{-2}\big)$ convergence analysis for deterministic multilevel gradient-regularized Newton methods employing such a pre-step level-selection criterion has remained unavailable. Establishing this rate will provide theoretical justification for employing multilevel methods in convex optimization, combining the fast convergence guarantees of full-space Newton methods with the potential computational savings of lower-dimensional corrections.

\textbf{Contributions.}
We develop multilevel gradient-regularized Newton methods for convex objectives with Lipschitz-continuous Hessians, exploiting lower-dimensional subspaces and the classical level-selection criterion \eqref{ineq:rmtr-conditions} (cf.~\cite{gratton2008recursive}). We show that these methods preserve the global $\Op\big(k^{-2}\big)$ function-value rate of their full-space counterpart. For two levels, \Cref{thm:two-level-mixed} establishes this rate under arbitrary coarse/fine alternation. For multiple levels with general restriction operators, \Cref{thm:multilevel-conditional} requires an explicit fine-iteration frequency condition, which is eventually satisfied by the algorithm (\Cref{sec:conditional-rate}) but may be restrictive during coarse-dominated phases. To eliminate this condition, we introduce \emph{row-nested restriction operators} (\Cref{sec:nested-operators}), whose structure guarantees monotonicity of reduced-gradient norms across levels. Exploiting this property, we establish the global $\Op\big(k^{-2}\big)$ rate without imposing any fine-iteration frequency condition (\Cref{thm:nested-operators}).

Our analysis builds on the ideas of \cite{mishchenko2023regularized}, but addresses a fundamental difficulty absent from full-space methods: coarse and fine corrections control different gradient norms, whose convergence properties cannot be directly transferred across levels. To overcome this, we introduce a level-dependent weighted gradient measure that provides uniform control across all possible coarse-to-fine, coarse-to-coarse, and fine-to-coarse transitions. This, together with the proposed class of row-nested restriction operators (\Cref{def:nested-restrictions}), constitutes the main theoretical novelty of our analysis, enabling us to preserve the convergence rate of the full-space method without imposing any condition on the frequency of fine-level iterations.

We support our theoretical analysis with numerical experiments comparing the proposed methods against the full-space gradient-regularized Newton method. The results demonstrate that the two-level and row-nested multilevel methods exhibit convergence behavior comparable to that of their full-dimensional counterpart, while achieving substantial reductions in computational time on large-scale instances. Furthermore, the experiments highlight the importance of the restriction operator structure, with row-nested hierarchies generally exhibiting more favorable convergence behavior than their non-nested counterparts. These findings suggest that row-nested restriction operators can offer an attractive choice from both theoretical and computational perspectives.

\subsection{Related work}
Early multilevel optimization frameworks were developed by Nash \cite{nash2000multigrid} and Lewis and Nash \cite{lewis2005model}, including implementations based on truncated-Newton methods. A common feature of subsequent multilevel methods \cite{gratton2008recursive,gross2009convergence,ho2019newton, calandra2021high} is the use of coarse-level eligibility conditions of the form \eqref{ineq:rmtr-conditions}. Wen and Goldfarb~\cite{WenGoldfarb2010} developed a line-search multigrid method establishing R-linear convergence under strong convexity and suitable conditions on the search directions. Calandra et al.~\cite{calandra2021high} extended higher-order regularization to multilevel optimization. Other multilevel Newton-type methods \cite{ho2019newton,tsipinakis2024multilevel,tsipinakis2023lowrank} establish global convergence or fast local rates under specific assumptions. However, unlike our approach, these Newton-type methods do not employ gradient-dependent Hessian regularization to handle potentially singular Hessians, nor do they establish the global $\Op\big(k^{-2}\big)$ function-value rate for general convex objectives.

Tsipinakis and Parpas \cite{tsipinakis2025regularized} developed the multilevel gradient-regularized Newton method achieving the global $\Op\big(k^{-2}\big)$ function-value rate for convex objectives. However, unlike the methods proposed in this paper, their method requires the look-ahead condition \eqref{ineq:condtion-trial-step}, introducing additional computational costs.  This condition avoids directly analyzing transitions between coarse and fine gradient measures and thus their proof cannot be applied to our setting. In contrast, our analysis preserves the global $\Op\big(k^{-2}\big)$ rate under classical level selection without requiring look-ahead trial steps.


Nested subspaces are widely used in Krylov-based trust-region methods, reduced-basis approximation, and multilevel optimization \cite{gould1999lanczos,binev2011convergence,gross2009convergence}. The row-nested restriction operators considered here additionally guarantee monotonicity of the restricted-gradient norms across levels. Moreover, the weighted-gradient analysis is closely related to multiple-Lyapunov-function techniques for switched systems \cite{branicky1998multiple,zhao2017stability}. In our setting, we combine a suitably weighted gradient measure with regularized Newton decrease bounds to establish a global $\Op\big(k^{-2}\big)$ convergence rate in function values under arbitrary coarse/fine alternation.

\subsection{Main assumptions and notation}
We consider the unconstrained optimization problem
\begin{equation*}\label{eq:optimization-problem}
    \min_{x\in\R^n} f(x),
\end{equation*}
where $f:\R^n\to\R$ is a twice continuously differentiable convex function. Given an initial point $x_0 \in \R^n$, our goal is to develop iterative methods achieving the global $\Op\big(k^{-2}\big)$ convergence rate in function values. Throughout, we impose the following assumptions.

\begin{assumption}\label{ass:bounded-sublevel-sets}
The initial sublevel set $    \mathcal L(x_0) \coloneqq  \{x\in\R^n \, : \,  f(x)\leq f(x_0)\}$ is bounded.
\end{assumption}

Since $f$ is continuous, $\mathcal L(x_0)$ is compact, ensuring the existence of a global minimizer $x^*\in\R^n$. We denote $f^*\coloneqq f(x^*)$.

\begin{assumption}\label{ass:lipschitz-hessians}
There exists an  $L>0$ such that
\[
    \norm{\hess f(x)-\hess f(y)} \leq L\norm{x-y}     \qquad \forall x,y\in\R^n.
\]
\end{assumption}

These assumptions are standard in the analysis of second-order optimization methods; see, e.g., \cite{nesterov2018lectures}. Throughout, $\norm{\cdot}$ denotes the Euclidean norm for vectors and the induced spectral norm for matrices. For convenience, we write
\[
    \grad f_k\coloneqq\grad f(x_k),     \qquad     \hess f_k\coloneqq\hess f(x_k).
\]
Moreover, by compactness of $\mathcal L(x_0)$, the constants
\begin{equation} \label{eq:H-and-D}
  H\coloneqq\sup_{x\in\mathcal L(x_0)} \norm{\hess f(x)}, \qquad 
  D\coloneqq\sup_{x\in\mathcal L(x_0)}
\norm{x-x^*}
\end{equation}
are finite.

\textbf{Organization.}
In \Cref{sec:two-level}, we introduce the two-level gradient-regularized Newton method and establish its convergence rate. In \Cref{sec:multilevel}, we extend the analysis to multiple coarse levels, establishing convergence guarantees under a fine-iteration frequency condition and, alternatively, without this condition using row-nested restriction operators introduced in \Cref{sec:nested-operators}. Numerical experiments comparing the proposed methods with full-space gradient-regularized Newton are presented in \Cref{sec:experiments}. Finally, \Cref{sec:conclusions} concludes the paper.

\section{The two-level Newton method with gradient regularization} \label{sec:two-level}

We consider the Newton method with gradient regularization that alternates between the fine and one coarse level. For the two-level algorithm we consider a single restriction operator that transfers the information from fine to coarse space, i.e., $R : \R^n \to \R^d$, where $d < n$ are the coarse space dimensions. Moreover, the prolongation operator is defined in a standard way, i.e., $P \coloneqq R^\top$. For the operators we take the standard assumptions (see for instance \cite{gratton2008recursive,ho2019newton}).

\begin{assumption}\label{ass:R}
Let $R\in\R^{d\times n}$ be a full-row-rank restriction operator, with $d\leq n$, and let $P=R^\top$ be its corresponding prolongation operator. We define
\begin{equation*}
    \omega\coloneqq\norm{R}.
\end{equation*}
\end{assumption}
Let $\sigma \in (0, \min\{1, \omega \})$ and $\varepsilon > 0$. We define the sets of coarse and fine iterations
\begin{equation*}
  \C \coloneq \{ k \in \N  \, : \,  \norm{R \gk} \geq \sigma \norm{\gk}, \,  \norm{R \gk} \geq \varepsilon\}, \quad \F \coloneq \N \setminus \C.
\end{equation*}
We can consider $\C$ as the set of effective coarse iterations and $\F$ the set of fine iterations.   
Thus, if $m \in \C$ then the algorithm will perform a coarse iteration otherwise a fine iteration. The two-level gradient-regularized Newton algorithm is presented in \Cref{alg:two-level}.

\begin{algorithm}[t]
\caption{Two-Level Regularized Newton (TLRN)}
\label{alg:two-level}
\begin{algorithmic}
\REQUIRE $x_0 \in \R^n, R : \R^n \to \R^d, \sigma \in (0, \min\{1, \omega \}), \varepsilon > 0$. 
\ENSURE Sequence of iterates $(x_0,x_1,x_2,\ldots)$.
\FOR{$k = 0, 1,  \ldots$}
\IF{$\nabla f_k = 0$} \STATE \textbf{stop} \ENDIF
\STATE Set {\begin{equation*} 
    x_{k+1} \coloneqq
\begin{cases}
x_k - P \left(R \Hk P + \sqrt{\frac{\omega^3 L \norm{R \gk}}{2}} I_d \right)^{-1} R \gk, & k \in \C, \\
x_k - \left( \Hk  + \sqrt{\frac{L \norm{\gk}}{2}} I_n \right)^{-1}  \gk,  & k \in \F.
\end{cases}
\end{equation*}}
\ENDFOR
\end{algorithmic}
\end{algorithm}

The following result will be needed for the analysis of \Cref{alg:two-level}; it has been proved in \cite[Lemma 3.3]{tsipinakis2025regularized} and \cite[Lemma 3]{mishchenko2023regularized}.

\begin{lemma}[\cite{tsipinakis2025regularized,mishchenko2023regularized}]
\label{lemma:f-non-increasing}
Suppose that Assumptions \ref{ass:lipschitz-hessians} and \ref{ass:R} hold. Let $x_k\in\R^n$ and suppose that $x_{k+1}$ is obtained by either the coarse or fine update of \Cref{alg:two-level}, whenever the corresponding update is well-defined. Then
\begin{align}
f(x_k)-f(x_{k+1})
&\geq \frac{\norm{R\nabla f_{k+1}}^2}
{8\sqrt{\omega^3 L\norm{R\gk}/2}}
&& \text{for a coarse update},
\label{ineq:reduction-coarse}\\
f(x_k)-f(x_{k+1})
&\geq \frac{\norm{\nabla f_{k+1}}^2}
{6\sqrt{L\norm{\gk}/2}}
&& \text{for a fine update}.
\label{ineq:reduction-fine}
\end{align}
\end{lemma}

\begin{corollary}\label{cor:f-nonincreasing} Suppose that Assumptions \ref{ass:lipschitz-hessians} and \ref{ass:R} hold and that the sequence ${(x_k)}_{k\in\N}$ is generated by \Cref{alg:two-level}. Then the sequence of function values ${(f(x_k))}_{k\in\N}$ is non-increasing, i.e.,
\[
f(x_{k+1}) \leq f(x_k), \qquad k\in\N.
\]
\end{corollary}

The following stability estimates for the coarse and fine corrections were also established in \cite{tsipinakis2025regularized} and \cite{mishchenko2023regularized}, and hold for all coarse and fine updates, respectively:
\begin{alignat}{2}
\norm{R \nabla f_{k+1}} &\leq 2 \norm{R \nabla f_k} &\qquad&  \text{for a coarse update}, \label{ineq:reduced-grad-bound}\\
\norm{\nabla f_{k+1}} &\leq 2 \norm{\nabla f_k} &\qquad& \text{for a fine update}. \label{ineq:full-grad-bound}
\end{alignat}

\subsection{Convergence analysis}
When $\C = \emptyset$, then \Cref{alg:two-level} coincides with the gradient regularization of Newton's method (see \cite{doikov2024gradient,mishchenko2023regularized}) which converges with the fast $\mathcal{O}\big(k^{-2}\big)$ rate. In what follows we show that this rate can be attained when the sequence ${(x_k)}_{k \in \N}$ is always generated from the coarse space, i.e.,
\[
x_{k+1} = x_k - P \left(R \Hk P + \sqrt{\frac{\omega^3 L \norm{R \gk}}{2}} I_d \right)^{-1} R \gk.
\]
For the purposes of this result, the condition $\norm{R\grad f_k}\geq\varepsilon$ is therefore not imposed. 

\begin{proposition}
\label{prop:rate-coarse-steps}
Suppose that Assumptions~\ref{ass:bounded-sublevel-sets} and \ref{ass:lipschitz-hessians}
hold, and let ${(x_k)}_{k\in\N}$ be generated by applying the coarse update
of Algorithm~\ref{alg:two-level} at every iteration. Assume that
\[
    \|R\nabla f_k\|\geq \sigma\|\nabla f_k\|,
    \qquad k\in\N.
\]
Then
\[
    f(x_k)-f^*=\mathcal{O}\big(k^{-2}\big).
\]
Moreover, the constant in the polynomial bound is proportional to
$\sigma^{-3}$.
\end{proposition}

\begin{proof}
Fix $k\in\N$ and set $ I^k\coloneqq \left\{ m\in\{0,\ldots,k-1\} \, : \, \|R\nabla f_{m+1}\|>\frac14\|R\nabla f_m\|
\right\}
$. We distinguish two cases.

\textbf{Case I:} $\big|I^k\big| \le k/2$.
By the definition of $I^k$ and \eqref{ineq:reduced-grad-bound},
\[
\|R\nabla f_{m+1}\|
\leq
\begin{cases}
\frac14\|R\nabla f_m\|, & m\notin I^k,\\[1ex]
2\|R\nabla f_m\|, & m\in I^k.
\end{cases}
\]
Therefore,
\[
    \|R\nabla f_k\|
    \leq
    \frac{2^{\left|I^k\right|}}{4^{k-\left|I^k\right|}}
    \|R\nabla f_0\|
    \leq
    2^{-k/2}\|R\nabla f_0\|.
\]
Using convexity and the coarse-level condition,
\[
\begin{aligned}
f(x_k)-f^*
&\leq D\|\nabla f_k\|
\leq \frac{D}{\sigma}\|R\nabla f_k\| 
\leq
\frac{\omega D}{\sigma}\,
2^{-k/2}\|\nabla f_0\|.
\end{aligned}
\]
Thus, in this case the function value gap decreases geometrically and,
in particular, is $\mathcal{O}\big(k^{-2}\big)$.

\textbf{Case II:} $\big|I^k\big| > k/2$.
Let $\tau\coloneqq \big|I^k\big|$, enumerate $    I^k=\{k_0,\ldots,k_{\tau-1}\},   \   k_0<\cdots<k_{\tau-1}<k,$
and set $k_\tau\coloneqq k$. By \eqref{ineq:reduction-coarse} and the definition
of $I^k$, for $\ell=0,\ldots,\tau-1$,
\[
\begin{aligned}
f(x_{k_\ell})-f(x_{k_\ell+1})
&\geq \frac{\|R\nabla f_{k_{\ell}+1}\|^{2}}{8(\omega^3L\|R\nabla f_{k_{\ell}}\|/2)^{1/2}} \ge
\frac{\|R\nabla f_{k_\ell}\|^{3/2}}{128(\omega^3L/2)^{1/2}} \geq
\frac{\sigma^{3/2}\|\nabla f_{k_\ell}\|^{3/2}}
     {128(\omega^3L/2)^{1/2}} .
\end{aligned}
\]
    Since \(I^k \subseteq\{0,\ldots,k-1\}\), it follows that
\(
k_\ell+1\le k_{\ell+1}\) for \( \ell=0,\ldots,\tau-1.
\)
    Moreover, by \eqref{ineq:reduction-coarse}, the sequence
${(f(x_j))}_{j\in\N}$ is non-increasing. Hence,
\begin{equation*}
  f(x_{k_{\ell}}) - f(x_{k_{\ell+1}}) \geq \frac{\sigma^{3/2}}{128(\omega^3L/2)^{1/2}}\|\nabla f_{k_\ell}\|^{3/2}.
\end{equation*}
Combining this with convexity, i.e., $f(x_{k_\ell}) - f^* \leq \langle \nabla f_{k_\ell}, x_{k_\ell} - x^* \rangle \le D \| \nabla f_{k_\ell} \|$, we obtain
\[
f(x_{k_{\ell+1}})-f^*
\leq
f(x_{k_\ell})-f^*
-C_{\mathcal C}\bigl(f(x_{k_\ell})-f^*\bigr)^{3/2},
\]
where
\[
    C_{\mathcal C}
    \coloneqq
    \frac{\sigma^{3/2}}
    {128(D^3\omega^3L/2)^{1/2}}.
\]
Hence, \Cref{prop:recurrence-rate} yields
\[
    f(x_k)-f^*
    \leq
    \frac{1}{C_{\mathcal C}^2(1+\tau/2)^2}
    \leq
    \frac{1}{C_{\mathcal C}^2(1+k/4)^2}
    =
    \mathcal{O}\!\left(\frac{1}{\sigma^3k^2}\right).
\]
 Combining the two cases proves
the result.
\end{proof}

\Cref{prop:rate-coarse-steps} analyzes the limiting case in which only coarse corrections are performed and shows that the coarse-only iteration retains the $\mathcal{O}\big(k^{-2}\big)$ rate of the full-space regularized Newton method. Its proof adapts the strategy of \cite[Theorem~1]{mishchenko2023regularized} to the reduced-gradient setting. A key feature of this analysis is that every iteration controls the same quantity, namely $\|R\nabla f_k\|$, so the corresponding estimates can be propagated consistently along the sequence.

The situation is fundamentally different for \Cref{alg:two-level}. Coarse iterations naturally control $\|R\nabla f_k\|$, whereas fine iterations control $\|\nabla f_k\|$, and neither of these quantities is necessarily contractive when the algorithm switches to the other level. Thus, the fact that both the coarse-only and fine-only schemes achieve an $\mathcal{O}\big(k^{-2}\big)$ rate does not by itself imply convergence, or the same rate, for an arbitrary mixture of the two iteration types. The main difficulty is therefore not the behavior within either level, but the transitions between them.

The following lemma provides the cross-level stability estimate needed to bridge this mismatch: although a coarse correction need not decrease the full gradient norm, it can increase it by at most a uniform factor. 

\begin{lemma} \label{lemma:bounded-gradient-in-coarse-level}
  Suppose that Assumptions \ref{ass:bounded-sublevel-sets}, \ref{ass:lipschitz-hessians}, and \ref{ass:R} hold and that \({(x_k)}_{k\in\mathbb N}\) is generated by Algorithm 2.1. Let 
\begin{equation*}
M_\C \coloneqq 1 + \frac{H \sqrt{2 \omega}}{\sqrt{L \varepsilon}},
\end{equation*}
where $H$ is defined in \eqref{eq:H-and-D}. Then, for any $k \in \C$ we have 
\begin{align*}
    \norm{\gkp} 
    & \leq M_\C \norm{\gk} \\
\end{align*}
\end{lemma}

\begin{proof}
  Denote 
  \begin{equation*}
      \alpha_k \coloneq \sqrt{\frac{\omega^3 L \norm{R \gk}}{2}}.
  \end{equation*}
  Next, we have 
  \[
\nabla f(x_{k+1})-\nabla f(x_k)
=
\int_0^1
\nabla^2 f\bigl(x_k+t(x_{k+1}-x_k)\bigr)
(x_{k+1}-x_k)\,\dd t.
\]
By \Cref{cor:f-nonincreasing}, \(x_k,x_{k+1}\in\mathcal L(x_0)\). Since \(f\) is convex, \(\mathcal L(x_0)\) is convex, and hence \(x_k+t(x_{k+1}-x_k)\in\mathcal L(x_0)\) for every \(t\in[0,1]\). Therefore,  
\begin{equation}\label{eq:boundgrad}
\|\nabla f(x_{k+1})-\nabla f(x_k)\|
\le H \|x_{k+1}-x_k\|.
\end{equation}
Then, by using the definitions of $\omega$ and $\alpha_k$ as well as  $k \in \C$ and the symmetric positive semi-definiteness of $R \Hk P$, we obtain
  \begin{align*}
    \norm{\gkp} & \leq \norm{\gkp - \gk} + \norm{\gk} \\
    & \leq H \norm{P \left(R \Hk P + \sqrt{\frac{\omega^3 L \norm{R \gk}}{2}} I_d\right)^{-1} R \gk} + \norm{\gk} \\ 
    &\leq H \omega \norm{\left(R \Hk P + \sqrt{\frac{\omega^3 L \norm{R \gk}}{2}} I_d\right)^{-1} R \gk} + \norm{\gk} \\
    & \leq H \omega \frac{\norm{R \gk}}{\alpha_k} + \norm{\gk} \\
    & \leq H \omega \frac{\sqrt{2}}{\sqrt{\omega^3 L \varepsilon}} \norm{R \gk} + \norm{\gk} \\
    & \leq \left(1 + \frac{H \sqrt{2 \omega}}{\sqrt{L \varepsilon}}\right) \norm{\gk}.
\end{align*}
\end{proof}

\Cref{lemma:bounded-gradient-in-coarse-level} shows that a coarse correction is stable with respect to the full gradient, in the sense that the gradient norm can increase by at most a uniformly bounded factor. Together with the level-specific stability estimates \eqref{ineq:reduction-coarse} and \eqref{ineq:reduction-fine}, this bound will allow us to construct a level-dependent weighted gradient measure that controls transitions between the two levels.

\begin{theorem} \label{thm:two-level-mixed}
  Let $f$ be convex and suppose that Assumptions  \ref{ass:bounded-sublevel-sets}, \ref{ass:lipschitz-hessians}, and~\ref{ass:R} hold. Then the sequence generated by Algorithm~\ref{alg:two-level} satisfies
\[
    f(x_k)-f^* = \mathcal{O} \left(\frac{1}{k^2}\right).
\]

\end{theorem}

  \begin{proof}
    To compare the reduced and full gradient norms across changes of level, we introduce a level-dependent weighted gradient measure. The weights are chosen so that all four possible coarse/fine transitions admit uniform contraction or stability bounds. Choose 
    \begin{equation*}
      a > 0,
      \qquad
    \rho\in\left(0,\frac{\sigma}{\omega M_C}\right),
    \qquad
    t \in \left(\rho\omega ,\frac{\sigma}{M_C}\right),
    \end{equation*}
    and set $b \coloneqq a t$.
    Since $\|R\|\|\gk\| \ge \|R\gk\| \ge \sigma \|\gk\|$ implies $\omega \ge \sigma$, we get
    \begin{equation*}
      q \coloneqq \max \left\lbrace \rho, \frac{t M_\C}{\sigma}, \frac{\omega \rho}{t} \right\rbrace < 1.
    \end{equation*}
    Moreover, define the scalar
    \begin{equation*}
      B \coloneqq \max \left\lbrace 2, \frac{2\omega}{t}, \frac{t M_\C}{\sigma} \right\rbrace,
    \end{equation*}
    and the sets
    \begin{align*}
\mathcal C^k & \coloneqq\mathcal C\cap\{0,\ldots,k-1\},
\qquad
\mathcal F^k\coloneqq\mathcal F\cap\{0,\ldots,k-1\},\\
    \mathcal I_{\C}^k(\rho)
    &\coloneqq
    \left\{
    m\in\C\cap\{0,\ldots,k-1\} \, : \,
    \norm{R\nabla f_{m+1}}>
    \rho\norm{R\nabla f_m}
    \right\},\\
    \mathcal I_{\F}^k(\rho)
    &\coloneqq
    \left\{
    m\in\F\cap\{0,\ldots,k-1\} \, : \, 
    \norm{\nabla f_{m+1}}>
    \rho\norm{\nabla f_m}
    \right\}.
    \end{align*}
    We set
    \[
        \mathcal I^k(\rho)
        \coloneqq
        \mathcal I_{\C}^k(\rho)
        \cup
        \mathcal I_{\F}^k(\rho).
    \]
    Now, choose \(
\theta\in\left(0,\frac{-\ln q}{\ln B-\ln q}\right)
\) and fix a $k \in \N$. We distinguish two cases. 

    \textbf{Case I:} $\big|\I^k(\rho)\big| \leq \theta k$. Define
    \begin{equation*}
        V_k \coloneqq 
        \begin{cases}
        a\|R \gk\|, & k \in \C,\\
        b\|\gk \|,& k\in \F,
        \end{cases}
    \end{equation*}
    For \(m\in (\mathcal I^k(\rho))^{\rm c} \coloneqq \{0,\ldots, k-1 \} \setminus \I^k(\rho) =    (\C^k \setminus \I_\C^k (\rho)) \cup (\F^k \setminus \I_\F^k(\rho))\), we show that \(V_{m+1}\le qV_m\).
    Let  $m \in \I^k(\rho)^{\rm c}$. Then, for $m$ and $m+1$ we have four possibilities:
    \begin{enumerate}[label=(\alph*)]
    \item $m \in \C^k \setminus \I_\C^k (\rho), \ m+1 \in \C$. Then,
    \begin{equation*}
      V_{m+1} \overset{m+1 \in \C}{=} a \norm{R \grad f_{m+1}} \overset{m \in \C \setminus \I_\C^k(\rho)}{\leq} a \rho \norm{R \grad f_m} \overset{m \in \C}{=} \rho V_m.
    \end{equation*}
    \item $m \in \C^k \setminus \I_\C^k (\rho), \ m+1 \in \F$. Then, with similar arguments we obtain
    \begin{align*}
    V_{m+1} &\overset{m+1 \in \F}{=} b \norm{\grad f_{m+1}}  \overset{\text{Lem.}~\ref{lemma:bounded-gradient-in-coarse-level}}{\leq} b M_\C \norm{\grad f_{m}}  \\ &\quad \  {\leq} \frac{b M_\C}{\sigma} \norm{R\grad f_{m}} \overset{m \in \C}{\leq} \frac{t M_\C}{\sigma} V_m.
    \end{align*}
    \item $m \in \F^k \setminus \I_\F^k (\rho), \ m+1 \in \F$. Then, 
    \begin{equation*}
      V_{m+1} \overset{m+1 \in \F}{=}  b \norm{\grad f_{m+1}} \overset{m \in \F \setminus \I_\F^k (\rho)}{\leq} b \rho \norm{\grad f_{m}} \overset{m \in \F}{=} \rho V_m.
    \end{equation*}
        \item $m \in \F^k \setminus \I_\F^k (\rho), \ m+1 \in \C$. Then,
        \begin{equation*}
          V_{m+1} \overset{m+1 \in \C}{=} a \norm{R \grad f_{m+1}} \leq a \omega \norm{\grad f_{m+1}} \overset{m \in \F \setminus \I_\F^k (\rho)}{\leq} a \omega \rho \norm{\grad f_{m}} \overset{m \in \F}{=} \frac{\omega \rho}{t} V_m.
        \end{equation*}
\end{enumerate}
Hence, 
\begin{equation}\label{ineq:V_m-contraction}
  V_{m+1} \leq q V_m \ \ \text{for any} \ \ m \in \I^k(\rho)^{\rm c}.
\end{equation}
Now let $m \in \I^k(\rho)$. We again consider four cases as above:
\begin{enumerate}[label=(\alph*)]
\item $m \in \I_\C^k (\rho), \ m+1 \in \C$. Then,
\begin{equation*}
  V_{m+1} = a \norm{R \grad f_{m+1}} \overset{\eqref{ineq:reduced-grad-bound}}{\le} 2 V_m.
\end{equation*}
\item $m \in \I_\C^k (\rho), \ m+1 \in \F$. Then, 
\begin{equation*}
V_{m+1} \leq \frac{t M_\C}{\sigma} V_m.
\end{equation*}
    \item $m \in \I_\F^k (\rho), \ m+1 \in \F$. Then,
    \begin{equation*}
        V_{m+1} = b \norm{ \grad f_{m+1}} \overset{\eqref{ineq:full-grad-bound}}{\le} 2 V_m.
    \end{equation*}
\item $m \in \I_\F^k (\rho), \ m+1 \in \C$. Then, 
\begin{equation*}
V_{m+1} \leq \frac{2 \omega}{t} V_m.  
\end{equation*}
\end{enumerate}
 As a result,
 \begin{equation} \label{ineq:V_m-I-recurrent-ineq}
  V_{m+1} \leq B V_m \ \ \text{for any} \ \ m \in \I^k(\rho).
 \end{equation}
 Let $\eta \coloneqq q^{1-\theta} B^\theta$. Thus, we have
 \begin{equation*}
 \ln\big(q^{1-\theta}B^\theta\big) = (1-\theta) \ln q + \theta \ln B = \ln q +  \theta (\ln B- \ln q),
 \end{equation*}
 and hence, by the definition of $\theta$, $\eta < 1$.
Combining now \eqref{ineq:V_m-contraction} and \eqref{ineq:V_m-I-recurrent-ineq} we obtain 
\begin{align}\label{ineq:V_k-final-bound}
  V_k \leq q^{k - \left|\I^k(\rho)\right|} B^{\left|\I^k(\rho)\right|} V_0 \leq \eta^k V_0.
\end{align}
Furthermore, for any $k \in \N$, we have that
\begin{equation} \label{ineq:grad-V_k-relation}
  \norm{\gk} \leq A V_k, \quad  A \coloneqq \max \left\lbrace  \frac{1}{a \sigma}, \frac{1}{b}  \right\rbrace.
\end{equation}
Then, by convexity and \Cref{ass:bounded-sublevel-sets}, we obtain
\begin{equation}
  f(x_k) - f^* \leq D \norm{\gk} \overset{\eqref{ineq:V_k-final-bound}, \eqref{ineq:grad-V_k-relation}}{\leq} \eta^k A D V_0 
  = \mathcal{O} \left(\frac{1}{k^2}\right),
\end{equation}
and thus the desired rate of convergence is established when $\big|\I^k(\rho)\big| \leq \theta k$.

    \textbf{Case II:} $\big|\I^k(\rho)\big| > \theta k$. Let $\tau \coloneqq \big|\I^k(\rho)\big|$ and enumerate $\I^k(\rho) = \{k_0, \ldots k_{\tau-1}\}$, with $
k_0<k_1<\cdots<k_{\tau-1} < k, \  k_\tau \coloneqq k$.
    Let $\ell \in \{0, \ldots, \tau - 1\}$. Thus, $k_\ell \in \I^k(\rho)$. If $k_\ell \in \I^k_\C(\rho)$, then
    \begin{align*}
    f(x_{k_\ell}) - f(x_{k_\ell + 1}) & \overset{\eqref{ineq:reduction-coarse}}{\geq} \frac{\|R \nabla f_{k_\ell + 1}\|^2}{8 \left( \omega^3 L / 2 \|R \nabla f_{k_\ell}\|\right)^{1/2}} \overset{k_\ell \in \I^k_\C(\rho)}{\geq} \frac{ \rho^2 \|R \nabla f_{k_\ell}\|^{3/2}}{8 \left( \omega^3 L / 2 \right)^{1/2}} \\ 
    & \overset{k_\ell \in \C}{\geq}  \left(\frac{\sigma}{\omega}\right)^{3/2}\frac{\rho^2 \| \nabla f_{k_\ell}\|^{3/2}}{8  \sqrt{L / 2} }.
    \end{align*}
    Using similar arguments, if $k_\ell \in \I^k_\F(\rho)$, then
    \begin{align*}
    f(x_{k_\ell}) - f(x_{k_\ell + 1}) \geq \frac{\rho^2 \norm{\gkl}^{3/2}}{6 \sqrt{L / 2}},
    \end{align*}
    and thus, since $\sigma \leq \omega$, we get
    \begin{align*}
      f(x_{k_\ell}) - f(x_{k_\ell + 1})  & \geq \min \left\lbrace  \left(\frac{\sigma}{\omega}\right)^{3/2}\frac{\rho^2 }{8  \sqrt{L / 2}  }, \frac{\rho^2 }{6 \sqrt{L / 2}} \right\rbrace \| \nabla f_{k_\ell}\|^{3/2} \\ 
      & = \left(\frac{\sigma}{\omega}\right)^{3/2}\frac{\rho^2 \| \nabla f_{k_\ell}\|^{3/2}}{8  \sqrt{L / 2} }
    \end{align*} 
    for all $\ell \in \{0, \ldots, \tau - 1\}$.
    Since \(I^k(\rho)\subseteq\{0,\ldots,k-1\}\), it follows that
\(
k_\ell+1\le k_{\ell+1}\) for \( \ell=0,\ldots,\tau-1.
\)
Following the same argument as in Case~II of \Cref{prop:rate-coarse-steps}, monotonicity of the function values and convexity yield
\[
    \Delta_{\ell+1}
    \leq
    \Delta_\ell-C\Delta_\ell^{3/2},
    \qquad
    \Delta_\ell\coloneqq f(x_{k_\ell})-f^*,
\]
where
\[
    C\coloneqq
    \frac{\sqrt{2}\rho^2}{8\sqrt{L}}
    \left(\frac{\sigma}{\omega D}\right)^{3/2}.
\]
Hence, \Cref{prop:recurrence-rate} and
$\tau > \theta k$ give
\[
    f(x_k)-f^*
    =\Delta_\tau
    \leq
    \frac{1}{C^2(1+\tau/2)^2}
    \leq
    \frac{1}{C^2(1+\theta k/2)^2}
    =
    \mathcal{O}\big(k^{-2}\big),
\]
which completes the proof of Case II.

Combining now the results of the two cases we have
\[
f(x_k)-f^* \leq \max \left\lbrace ADV_0 \eta^k,   \dfrac{1}{C^2(1+\theta k/2)^2} \right\rbrace =
    \mathcal{O}\big(k^{-2}\big),
\]
as required.
 \end{proof}

\Cref{thm:two-level-mixed} shows that the use of coarse corrections does not deteriorate the global
convergence rate of the regularized Newton method. In particular, Algorithm~\ref{alg:two-level}
achieves the $\mathcal{O}\big(k^{-2}\big)$ rate in function values while allowing arbitrary switching
between coarse and fine iterations according to the level-selection criterion. A key feature of the
analysis is that coarse and fine iterations naturally control different quantities, namely
$\|R\nabla f_k\|$ and $\|\nabla f_k\|$, respectively. The proof resolves this mismatch through a
level-dependent weighted gradient measure that provides uniform control across all possible
coarse/fine transitions: it contracts on iterations outside $\mathcal I^k(\rho)$, while its growth
on the remaining iterations is uniformly bounded. Consequently, the analysis does not require the
full gradient norm to decrease at every coarse iteration; uniform stability of the coarse
corrections is sufficient. This provides a mechanism for preserving convergence rates in multilevel
schemes even when different iteration types control different measures of stationarity.

\section{The multilevel Newton method with gradient regularization}  \label{sec:multilevel}

In this section we consider the gradient regularization of Newton's method that alternates between $\ell$ coarse levels (plus the fine level) when generating the updates. In particular, we consider $\ell$ restriction operators as in \cite{gratton2008recursive}, but unlike their recursive framework, our operators map directly between each coarse level and the full space, allowing coarse corrections to be transferred directly to the fine level. More precisely, for each $i \in \mathcal{J} \coloneqq \{1, \ldots, \ell \}$, let $R_i : \R^n \to \R^{n_i}$ be a restriction operator, where $n_1 < \cdots < n_\ell < n$. Further, $P_i \coloneqq R^\top_i$ is the prolongation operator at level $i \in \mathcal{J}$. We will see that this nonrecursive formulation simplifies both the algorithm and its convergence analysis while retaining the same $\Op\big(n_i^2+n_i^3\big)$ computational cost at the $i$th coarse level.

Accordingly, \Cref{ass:R} is now replaced by the following assumption.
\begin{assumption}\label{ass:R_i}
For each $i\in\J$, let $R_i\in\R^{n_i\times n}$ be a full-row-rank restriction operator, with $n_i\leq n$, and let $P_i=R_i^\top$ be its corresponding prolongation operator. We define
\[
    \omega_i\coloneqq\norm{R_i},
    \qquad
    \zeta\coloneqq\max_{i\in\J}\omega_i.
\]
\end{assumption}
Let $ \sigma\in \left(0, \min\left\{1,\min_{i\in\mathcal J}\omega_i\right\}\right)$ and $\varepsilon > 0$. For each $i \in \J$, the sets of coarse and fine iterations are now defined as follows. First define the eligible sets
\[
\mathcal E_i \coloneqq \left\{ k\in\mathbb N \, : \, \|R_i\nabla f_k\|\ge\sigma\|\nabla f_k\|, \, \|R_i\nabla f_k\|\ge \varepsilon \right\}.
\]
The set \(\mathcal E_i\) contains the iterations for which level \(i\) is eligible. For the multilevel method to become efficient, we need to select the lowest-dimensional eligible level
\begin{align*}
  \C_i \coloneq \mathcal{E}_i \setminus \bigcup_{j < i} \mathcal{E}_j, \qquad i \in \J. 
\end{align*}
Then the sets $\mathcal C_i$ are pairwise disjoint and, whenever
$\{i\in\mathcal J \, : \, k\in\mathcal E_i\}\neq\emptyset$,
\[
k\in\mathcal C_i
\quad\Longleftrightarrow\quad
i=i^*(k)\coloneqq
\min\{j\in\mathcal J \, : \, k\in\mathcal E_j\}.
\]
Moreover, we set 
\begin{equation*}
  \F^{(\ell)} \coloneqq \N \setminus \C^{(\ell)}, \qquad \C^{(\ell)} \coloneqq \bigcup_{i \in \J} \C_i.
\end{equation*} 
The multilevel gradient-regularized Newton method is presented in \Cref{alg:multilevel}. 

\begin{algorithm}[tb]
\caption{MultiLevel Regularized Newton (MLRN)}
\label{alg:multilevel}
\begin{algorithmic}
\REQUIRE $x_0 \in \R^n, \ell \in \N, R_i : \R^n \to \R^{n_i}$ for $i \in \J$, $\sigma\in \left(0, \min\left\{1,\min_{i\in\mathcal J}\omega_i\right\}\right), \varepsilon > 0$. 
\ENSURE Sequence of iterates $(x_0,x_1,x_2,\ldots)$.
\FOR{$k = 0, 1,  \ldots$}
\IF{$\nabla f_k = 0$} \STATE \textbf{stop} \ENDIF
\STATE Set \begin{equation*}
x_{k+1} \coloneqq
\begin{cases}
x_k - P_i
\left(
R_i \Hk P_i
+
\sqrt{\dfrac{\omega_i^3 L \norm{R_i \gk}}{2}}\, I_{n_i}
\right)^{-1}
R_i \gk,
& k \in \C_i, i\in\J, \\[2ex]
x_k -
\left(
\Hk
+
\sqrt{\dfrac{L\norm{\gk}}{2}}\, I_n
\right)^{-1}
\gk,
& k \in \F^{(\ell)}.
\end{cases}
\end{equation*}
\ENDFOR
\end{algorithmic}
\end{algorithm}

The next lemma shows that \Cref{alg:multilevel} generates a non-increasing sequence of function values.
It also shows that it guarantees the reduced-gradient stability of a coarse correction step. These
results arise by \cite{tsipinakis2025regularized}.

\begin{lemma}
\label{lemma:multilevel-coarse-estimates}
Suppose that Assumptions~\ref{ass:lipschitz-hessians} and
\ref{ass:R_i} hold, and let ${(x_k)}_{k\in\N}$ be generated by
\Cref{alg:multilevel}. Then, for any $k\in\C_i$, $i\in\J$,
\begin{align}
f(x_k)-f(x_{k+1})
&\geq
\frac{\|R_i\nabla f_{k+1}\|^2}
{8\sqrt{\omega_i^3L\|R_i\nabla f_k\|/2}},
\label{ineq:multilevel-reduction}\\
\|R_i\nabla f_{k+1}\|
&\leq
2\|R_i\nabla f_k\|.
\label{ineq:multilevel-reduced-gradient-stability}
\end{align}
\end{lemma}

\begin{proof}
The proof is outlined in \Cref{sec:multilevel-coarse-estimates}.
\end{proof}

\Cref{lemma:multilevel-coarse-estimates} guarantees that \( f(x_{k+1})\le f(x_k), \  k\in \mathcal C_i \), and by \eqref{ineq:full-grad-bound}, ${(f(x_k))}_{k \in \N}$ is non-increasing. In particular, this means that the multilevel iterates stay in $\mathcal{L}(x_0)$.
Finally, in \Cref{lemma:finite-coarse} we show that \Cref{alg:multilevel} performs only finitely many coarse iterations.

\subsection{Conditional rate of convergence} \label{sec:conditional-rate}

Let $\rho \in (0,1)$. Similar to the analysis of the two-level method, we will need the following definitions,
\begin{align*}
\C_i^k  &\coloneqq \C_i \cap \{0,\ldots ,k-1\},
\ \  \C^{(\ell), k} \coloneqq \bigcup_{i \in \J} \C_i^k, \\
\mathcal F^{(\ell), k} &\coloneqq \mathcal F^{(\ell)} \cap \{0,\ldots,k-1\}, \\
    \mathcal I_{\C_i}^k(\rho) & \coloneqq \left\{ m\in\C^k_i \, : \, \norm{R_i \nabla f_{m+1}}> \rho\norm{R_i \nabla f_m} \right\}, \\
    \mathcal I_{\F^{(\ell)}}^k(\rho) & \coloneqq \left\{ m\in\F^{(\ell), k} \, : \, \norm{\nabla f_{m+1}}> \rho\norm{\nabla f_m}
    \right\},
    \end{align*}
    where $i \in \J$. We further set
    \[
        \I_\ell^k(\rho) \coloneqq \left( \bigcup_{i \in \J} I_{\C_i}^k(\rho) \right)  \cup \I_{\F^{(\ell)}}^k(\rho).
    \]
Hence, the disjointness of the sets $\C_i$, $i \in \J$, and $\F^{(\ell)}$ gives
 \begin{equation*}
    \big|\I_\ell^k(\rho)\big| = \sum_{i=1}^\ell \big| I_{\C_i}^k(\rho)\big|  + \big|\I_{\F^{(\ell)}}^k(\rho)\big|.
 \end{equation*}

The following lemma shows the $\Op\big(k^{-2}\big)$ rate of \Cref{alg:multilevel} when $\big|\I_\ell^k(\rho)\big|$ is sufficiently large.
\begin{lemma} \label{lemma:case-II}
  Suppose that Assumptions \ref{ass:bounded-sublevel-sets}, \ref{ass:lipschitz-hessians}, and \ref{ass:R_i} hold. Let $k \in \N$ and $\theta, \rho \in (0,1)$. If $\big|\I_\ell^k (\rho)\big| > \theta k$, then the sequence generated by \Cref{alg:multilevel} satisfies  
    \begin{equation*}
      f(x_k) - f^* \leq \frac{1}{C_\ell^2 \left(1 + \frac{\theta k}{2}\right)^2}, \qquad C_\ell \coloneqq \sqrt{ \frac{\rho^4 \sigma^3}{32 L \zeta^3 D^3}}.
    \end{equation*}
\end{lemma}

\begin{proof}
The proof is given in \Cref{sec:case-II}.
\end{proof}

Further, the coarse correction is also stable with respect to the full gradient. In particular, by \Cref{lemma:bounded-gradient-in-coarse-level} there exist $M_{\C_i} \in (0, \infty)$, $i \in \J$, such that 
\begin{equation*}
      \norm{\gkp}   \leq M_{\C_i} \norm{\gk}, \qquad k \in \C_i.
\end{equation*}
In what follows we show that \Cref{alg:multilevel} converges with $\Op\big(k^{-2}\big)$ when it performs a sufficiently large numbersufficiently of fine steps.
\begin{theorem}
\label{thm:multilevel-conditional}
Suppose that Assumptions~\ref{ass:bounded-sublevel-sets}, \ref{ass:lipschitz-hessians}, and \ref{ass:R_i} hold. Let
\[
    M_{\C^{(\ell)}}
    \coloneqq
    \max\left\{1,\max_{i\in\J}M_{\C_i}\right\},
    \qquad
    h\coloneqq\log_2 M_{\C^{(\ell)}}.
\]
Let $\theta\in(0,2/3)$ and choose
\[
    a\in
    \left(
        \frac{h+3\theta}{h+2},1
    \right).
\]
Suppose also that, for all sufficiently large $k$,
\[
    \big|\F^{(\ell),k}\big|\geq ak.
\]
Then the sequence generated by \Cref{alg:multilevel} satisfies
\[
    f(x_k)-f^*=\mathcal O\big(k^{-2}\big).
\]
\end{theorem}

\begin{proof}
Fix a sufficiently large $k\in\N$ such that $\big|\F^{(\ell),k}\big|\geq ak$, and set $    \I_\ell^k
\coloneqq \I_\ell^k(1/4)$, $ \I_{\F^{(\ell)}}^k \coloneqq \I_{\F^{(\ell)}}^k(1/4).$ If
$\big|\I_\ell^k\big|>\theta k$, then \Cref{lemma:case-II} yields
\[
    f(x_k)-f^*=\mathcal O\big(k^{-2}\big).
\]

It remains to consider the case $|\I_\ell^k|\leq\theta k.$ For every $m\in\{0,\ldots,k-1\}$, exactly
one of the following possibilities occurs:
\begin{enumerate}[label=(\alph*)]
    \item If $m\in\C^{(\ell),k}$, then
    \(
        \|\nabla f_{m+1}\|
        \leq
        M_{\C^{(\ell)}}\|\nabla f_m\|.
    \)

    \item If $m\in\I_{\F^{(\ell)}}^k$, then, by
    \eqref{ineq:full-grad-bound},
    \(
        \|\nabla f_{m+1}\|
        \leq
        2\|\nabla f_m\|.
    \)

    \item If
    $m\in\F^{(\ell),k}\setminus\I_{\F^{(\ell)}}^k$,
    then
    \(
        \|\nabla f_{m+1}\|
        \leq
        \frac14\|\nabla f_m\|.
    \)
\end{enumerate}

Since
\(
    \big|\C^{(\ell),k}\big|+\big|\F^{(\ell),k}\big|=k,
\)
iterating these bounds gives
\[
\begin{aligned}
\|\nabla f_k\| &\leq M_{\C^{(\ell)}}^{\,k-\left|\F^{(\ell),k}\right|} 2^{\left|\I_{\F^{(\ell)}}^k\right|}
\left(\frac14\right)^{ \left|\F^{(\ell),k}\right|-\left|\I_{\F^{(\ell)}}^k\right|} \|\nabla f_0\|\\
&= M_{\C^{(\ell)}}^{\,k-\left|\F^{(\ell),k}\right|} 2^{-2\left|\F^{(\ell),k}\right| +3\left|\I_{\F^{(\ell)}}^k\right|} \|\nabla f_0\|.
\end{aligned}
\]
Since
\(
    \big|\F^{(\ell),k}\big|\geq ak \   \text{and} \ 
    \big|\I_{\F^{(\ell)}}^k\big|
    \leq \big|\I_\ell^k\big|
    \leq\theta k,
\)
we obtain
\[
    \|\nabla f_k\|
    \leq
    \beta^k\|\nabla f_0\|,
\]
where
\[
    \beta
    \coloneqq
    M_{\C^{(\ell)}}^{1-a}2^{-2a+3\theta}.
\]
Moreover,
\[
    \log_2\beta
    =
    h+3\theta-a(h+2)<0
\]
by the choice of $a$, and hence $\beta<1$.

Therefore, by convexity and
Assumption~\ref{ass:bounded-sublevel-sets},
\[
    f(x_k)-f^*
    \leq
    D\|\nabla f_k\|
    \leq
    D\beta^k\|\nabla f_0\|
    =
    \mathcal O\big(\beta^k\big)
    =
    \mathcal O\big(k^{-2}\big).
\]
Combining the two cases proves the result.
\end{proof}

\Cref{thm:multilevel-conditional} provides a quantitative condition under which the multilevel
scheme retains the convergence rate of the underlying full-space method. Although this rate need not
be inherited by an arbitrary multilevel realization, the theorem shows that the \(\mathcal
O\big(k^{-2}\big)\) rate is preserved whenever a sufficiently large fraction of the iterations is performed
on the fine level. Such preservation of complexity has also been established for other multilevel
high-order regularization frameworks; see, e.g., \cite{calandra2021high} for the multilevel ARC method and possibly non-convex functions. To the best of our knowledge, \Cref{thm:multilevel-conditional} is the first result showing that a classical multilevel gradient-regularized Newton method can retain the global \(\mathcal O\big(k^{-2}\big)\) function value rate of the full-space method. This is possible to achieve under an explicit quantitative condition on the proportion of fine-level iterations. 

We emphasize that the fine-iteration frequency condition is not restrictive. Indeed, the eligibility condition $\norm{R_i \grad f_k} \ge \varepsilon$ implies that only finitely many coarse iterations can occur, see \Cref{lemma:finite-coarse} for details. Consequently, the algorithm eventually performs exclusively fine-level iterations, and the required frequency condition is automatically satisfied for all sufficiently large $k$. Nevertheless, this condition is not computationally irrelevant. For instance, for large $k$ and \(\theta = 0.1\), \(M_{\mathcal
C^{(\ell)}}=2\) and \(M_{\mathcal C^{(\ell)}}=2000\) yield approximately \(a>0.44\) and \(a>0.87\),
respectively. This indicates that arbitrarily long sequences of coarse corrections cannot not occur before the algorithm enters the fine-only regime. 
In \Cref{sec:rate-with-nested} we develop analysis establishing an explicit convergence bound valid at every iteration, without relying on the eventual transition to the fine-only regime, and which also allows the method to perform arbitrarily long sequences of coarse corrections.

\subsection{Row-nested restriction operators} \label{sec:nested-operators}

In the previous subsection, we established an $\mathcal O\big(k^{-2}\big)$ convergence rate under the condition that a sufficiently large fraction of the iterations is performed on the fine level. Such a condition may lead to a large number of computationally expensive fine-level iterations before \Cref{alg:multilevel} enters its fine-level regime. Here we introduce an additional structure on the restriction operators and discuss its most important properties. This new structure will, later, allow us to show a global $\mathcal O\big(k^{-2}\big)$ rate for \Cref{alg:multilevel} without this condition.

\begin{definition}[Row-nested restriction operators]
\label{def:nested-restrictions}
Let $R_i\in\R^{n_i\times n}$, $i\in\J$, with
\[
    n_1<n_2<\cdots<n_\ell<n.
\]
The family $(R_i)_{i=1}^{\ell}$ is said to be
\emph{row-nested} if, for every $i=1,\ldots,\ell-1$, there exists
\[
    \widehat R_i\in\R^{(n_{i+1}-n_i)\times n}
\]
such that
\[
    R_{i+1}
    =
    \begin{bmatrix}
        R_i\\
        \widehat R_i
    \end{bmatrix}.
\]
\end{definition}

Nested hierarchies of coarse spaces are standard in multigrid and have also been employed in multilevel optimization; see, e.g., \cite{gross2009convergence, kocvara2016multigrid}. Here we impose a more specific algebraic structure: all restriction operators act on the same fine space, and the rows defining a lower-dimensional restriction are retained at every subsequent level. This row-nested construction yields the monotonicity \(\|R_i v\|\le \|R_jv\|\) for \(i<j\), which will play a central role in our convergence analysis.

For every $i<j$ there exists
$\widehat R_{ij}\in\R^{(n_j-n_i)\times n}$ such that
\[
    R_j
    =
    \begin{bmatrix}
        R_i\\
        \widehat R_{ij}
    \end{bmatrix}.
\]
Consequently, for every $v\in\R^n$,
\[
    \|R_jv\|^2
    =
    \|R_iv\|^2+\big\|\widehat R_{ij}v\big\|^2,
\]
and therefore,
\begin{equation}
\label{ineq:nested-restriction-monotonicity}
    \|R_iv\|\leq\|R_jv\|,
    \qquad 1\leq i<j\leq\ell.
\end{equation}
Further, \Cref{def:nested-restrictions} implies nestedness in multilevel spaces. By the definition of $R_j$ we have
\[
P_j=R_j^\top = 
\begin{bmatrix}
P_i & \widehat R_{ij}^{\top}
\end{bmatrix}.
\]
Therefore, since \(n_i<n_j\),
\begin{equation} \label{eq:nested-spaces-property}
  \operatorname{range}(P_i)
\subset
\operatorname{range}(P_j),
\qquad i<j,
\end{equation}
and thus, as in classical multilevel methods with nested subspaces, each successive level enriches the
previous coarse correction space with additional directions. However, the stronger row-nested structure
additionally implies the norm monotonicity \eqref{ineq:nested-restriction-monotonicity}. There also is a
corresponding reverse nesting of null spaces, namely
\[
\ker(R_j)\subset\ker(R_i),
\qquad i<j,
\]
where the inclusion is strict since \(\dim\ker(R_j)=n-n_j<n-n_i=\dim\ker(R_i)\). Thus, larger coarse levels miss fewer directions of the fine-space gradient. 

\Cref{def:nested-restrictions} yields properties that significantly simplify the multilevel algorithm and provides us with further insights. In particular,
\[
    \omega_1\leq\omega_2\leq\cdots\leq\omega_\ell=\zeta,
\]
and the admissible range for \(\sigma\) simplifies to $\sigma\in(0,\min\{1,\omega_1\}).$ Moreover, \eqref{ineq:nested-restriction-monotonicity} implies that
the eligible sets satisfy
\[
    \mathcal E_1\subseteq\mathcal E_2
    \subseteq\cdots\subseteq\mathcal E_\ell, \qquad
    \bigcup_{j<i}\mathcal E_j=\mathcal E_{i-1}, \quad i = 2, \ldots, \ell.
\]
Hence,
\[
    \mathcal C_1=\mathcal E_1,
    \qquad
    \mathcal C_i
    =
    \mathcal E_i\setminus\mathcal E_{i-1},
    \quad i=2,\ldots,\ell,
\]
and,
\[
\mathcal C^{(\ell)} = \bigcup_{i=1}^{\ell}\mathcal C_i = \mathcal E_\ell,
\qquad
\mathcal F^{(\ell)} = \mathbb N\setminus\mathcal E_\ell.
\]
Therefore, \(k\in\mathcal C_i\) means that level \(i\) is the lowest-dimensional level satisfying the eligibility criterion at iteration \(k\). Every level \(j>i\) is then automatically eligible, whereas no level \(j<i\) is.

\begin{remark}
The nested structure also induces a hierarchy of reduced Hessian
matrices. Indeed, for $i<j$,
\[
R_j\nabla^2 f(x)P_j
=
\begin{bmatrix}
R_i\nabla^2 f(x)P_i
&
R_i\nabla^2 f(x)\widehat R_{ij}^{\top}\\
\widehat R_{ij}\nabla^2 f(x)P_i
&
\widehat R_{ij}\nabla^2 f(x)\widehat R_{ij}^{\top}
\end{bmatrix}.
\]
Hence, $R_i\nabla^2 f(x)P_i$ is a principal submatrix of $R_j\nabla^2 f(x)P_j$. Consequently, by the Cauchy interlacing theorem \cite{horn2012matrix}, their eigenvalues satisfy
\[
\lambda_r\!\left(R_j\nabla^2 f(x)P_j\right)
\le
\lambda_r\!\left(R_i\nabla^2 f(x)P_i\right)
\le
\lambda_{r+n_j-n_i}\!\left(R_j\nabla^2 f(x)P_j\right),
\]
for \(r=1,\ldots,n_i\), where the eigenvalues are ordered increasingly. Moreover, if $f$ is convex, then
\[
R_i\nabla^2 f(x)P_i\succeq0,
\qquad i\in\J,
\]
while $\nabla^2f(x)\succ0$ together with the full row rank of $R_i$
implies
\[
R_i\nabla^2 f(x)P_i\succ0.
\]
Thus, increasing the level enriches the reduced model while retaining the curvature information
contained in the preceding levels. However, note that the principal-submatrix relation concerns the unregularized reduced Hessians; the regularized matrices need not be nested because the regularization parameter depends on the level.
\medskip
\end{remark}

\begin{remark}
\label{rem:nested-orthonormal}
An important special case of Definition 3.4 arises when the restriction operators have orthonormal rows, i.e.,
\[
R_iR_i^\top=I_{n_i},\qquad i\in\mathcal J.
\]
Then
\[
Q_i\coloneqq P_iR_i=R_i^\top R_i
\]
is the orthogonal projector onto \(\operatorname{range}(P_i)\). Moreover, \eqref{eq:nested-spaces-property} implies
\[
Q_iQ_j=Q_jQ_i=Q_i,\qquad Q_i\preceq Q_j,\qquad i<j.
\]
Thus, increasing the level enlarges the projected subspace while retaining all directions represented at the preceding levels. 

A simple and computationally efficient construction satisfying \Cref{ass:R_i}, \Cref{def:nested-restrictions}, and the orthonormal-row condition considered above is obtained by selecting nested subsets of the rows of the identity matrix. Specifically, let $S_1\subset\cdots\subset S_\ell\subset\{1,\ldots,n\}$ with $|S_i|=n_i$, and let $R_i$ consist of the rows of $I_n$ indexed by $S_i$. For each $i=1,\ldots,\ell-1$, choosing $\widehat R_i$ as the rows of $I_n$ indexed by $S_{i+1}\setminus S_i$ yields
\[
R_{i+1}=
\begin{bmatrix}
R_i\\
\widehat R_i
\end{bmatrix}.
\]
Since the selected rows are distinct standard basis vectors, $R_iR_i^\top=I_{n_i}$ for every $i\in\mathcal J$. Thus, \Cref{ass:R_i} is satisfied with $\omega_i = 1$, $i \in \J$, and further $\sigma \in (0,1)$.
\medskip
\end{remark}

\begin{remark}
A general construction satisfying both Assumption 3.1 and Definition 3.4 is obtained by choosing any full row-rank restriction operator $R_\ell\in\R^{n_\ell\times n}$ and defining $R_i$ as its first $n_i$ rows, for $i=1,\ldots,\ell-1$. Indeed, the full row-rank of $R_\ell$ guarantees that $\operatorname{rank}(R_i)=n_i$, while the construction yields $R_{i+1}=\left[R_i^\top,\widehat R_i^\top\right]^\top$, where $\widehat R_i$ consists of rows $n_i+1,\ldots,n_{i+1}$ of $R_\ell$. This construction accommodates, for instance, Gaussian matrices, suitable sparse sketching matrices, and coordinate restrictions.
\end{remark}

\subsection{Convergence analysis using the row-nested restriction operators} \label{sec:rate-with-nested}
We now establish an $\mathcal O\big(k^{-2}\big)$ convergence rate for \Cref{alg:multilevel} under the row-nested structure of \Cref{def:nested-restrictions}, without imposing any condition on the frequency of fine-level iterations. The analysis extends the weighted-gradient argument developed for the two-level method. An additional difficulty arises from transitions between different coarse levels, since the corresponding updates control distinct reduced-gradient norms $\|R_i\nabla f_k\|$. For general restriction operators, these quantities need not be comparable across levels. The row-nested structure provides the necessary norm comparisons, which, together with suitable level-dependent weights, allow us to control transitions among all coarse and fine levels.

\begin{theorem} \label{thm:nested-operators}
  Let f be convex and suppose that Assumptions~\ref{ass:bounded-sublevel-sets} and~\ref{ass:lipschitz-hessians} hold. If the sequence \({(x_k)}_{k \in \N}\) is generated by \Cref{alg:multilevel} with row-nested operators for which \Cref{ass:R_i} hold, then
  \begin{equation*}
    f(x_k) - f^* = \Op\big(k^{-2}\big).
  \end{equation*}
\end{theorem}

\begin{proof}
  The proof makes use of the set $\I_\ell^k(\rho)$ and as previously we split the analysis into two cases, according to the magnitude of $\I_\ell^k(\rho)$. 
  
  Recall that, due to the row-nestedness, $\zeta = \omega_\ell$, and define $\bar{\omega} \coloneqq \max\{1, \omega_\ell \}$. We choose
  \begin{align*}
    \theta \in \left(0, \frac{1}{1+\ell}\right), \qquad r \in \left(0, \min \left\lbrace \frac{1}{\Lambda}, B \right\rbrace \right), \qquad \rho \in \left(0, \frac{\Lambda}{\bar{\omega}} r^{1+\ell} \right),
  \end{align*}
  where 
  \begin{equation*}
    \Lambda \coloneqq \frac{\bar{\omega} M_{\C^{(\ell)}}}{\sigma}, \qquad B \coloneqq 
    \exp \left( - \frac{(1-\theta) \ln(\Lambda) + \theta \ln(2\bar{\omega})}{1- (1+\ell)\theta}\right),
  \end{equation*}
  and note that $\Lambda \geq 1$ and $B, r, \rho \in (0, 1)$. Moreover, for $i \in \J$, we define
  \[
  a_i \coloneqq r^{i-1}, \qquad b \coloneqq r^\ell,
  \]
  therefore $1=a_1 > \cdots > a_\ell>0$, and note also that $    \frac{a_j}{a_i} = r^{j-i} > 1$, if $j<i$, and $r^{j-i} < 1$, if $j>i$. 
  
  Let us fix a $k \in \N$. If $\big|\I_\ell^k(\rho)\big| > \theta k$, then \Cref{lemma:case-II} provides the
  desired rate of convergence. Assume now that $\big|\I_\ell^k(\rho)\big| \leq \theta k$. Define the weight
  gradient measure for $\ell$ coarse levels
      \begin{equation*}
        V_k \coloneqq
        \begin{cases}
        a_i \|R_i \gk\|, & k \in \C_i, \ i \in \J,\\
        b\|\gk \|,& k\in \F^{(\ell)}.
        \end{cases}
    \end{equation*}
    Thus, by the definition of $a_i$ and $V_k$,
    \begin{equation}\label{ineq:grad-V-relation}
      \norm{\grad f_k} \leq \max \left\lbrace \frac{1}{a_1 \sigma}, \ldots, \frac{1}{a_\ell \sigma}, \frac{1}{b} \right\rbrace V_k = \max \left\lbrace \frac{1}{a_\ell \sigma}, \frac{1}{b} \right\rbrace V_k.
    \end{equation}
    Next, it holds
    \begin{equation*}
      \big(\I_\ell^k(\rho)\big)^{\rm c} \coloneqq \{0, 1, \ldots, k-1\} \setminus \I_\ell^k(\rho) = \left(\bigcup_{i \in \J} C_i^k \setminus \I_{\C_i}^{k} (\rho) \right) \cup \left(\F^{(\ell), k} \setminus \I_{\F^{(\ell)} }^{k} (\rho) \right). 
    \end{equation*}

    We first examine the case $m \in (\I_\ell^k(\rho))^{\rm c}$, where, in particular, we will show that $V_k$ is a contraction. To achieve this, in comparison to the proof of \Cref{thm:two-level-mixed}, we need to examine an extra case which concerns the transitions between the $\ell$ coarse levels: 
    \begin{enumerate}
      \item $m \in \C_i^k \setminus \I_{\C_i}^{k}$, $m+1 \in \C_j$. We analyze this case by splitting into additional three subcases according to the ordering of the coarse levels: 
       \begin{enumerate}
        \item $j < i$. Then, by the row-nestedness \eqref{ineq:nested-restriction-monotonicity}, we have
        \begin{align*}
          & V_{m+1}  = a_j\norm{R_j \grad f_{m+1}} \leq a_j\norm{R_i \grad f_{m+1}} \\ 
          & \quad \ \  \overset{k \notin \I_{\C_i}^{k}}{\leq} \rho a_j \norm{R_i \grad f_{m}} = r^{j-i} \rho V_k \leq r^{1-\ell} \rho V_k,
        \end{align*}
        where the last inequality uses the fact that $x \mapsto r^x$ is non-increasing for $r \in (0,1)$.
        \item $i < j$. Then, 
        \begin{align*}
          V_{m+1} & = a_j\norm{R_j \grad f_{m+1}} \leq a_j \bar{\omega} \norm{\grad f_{m+1}} \\
          & \leq a_j \bar{\omega} M_{\C^{(\ell)}}\norm{\grad f_{m}} \overset{k \in \C_i^k}{\leq} \Lambda r^{j-i} V_m \leq \Lambda r V_m,
        \end{align*}
        where we have use the uniform stability of the coarse corrections, i.e., $ \norm{\gkp}   \leq M_{\C^{(\ell)}} \norm{\gk}$ with  $M_{\C^{(\ell)}}$ defined in \Cref{thm:multilevel-conditional}.
        \item $i=j$. Then we immediately have
        \begin{equation*}
          V_{m+1} \leq \rho V_m.
        \end{equation*}
      \end{enumerate} 
      Combining the above subcases yields
      \begin{equation*}
        V_{m+1} \leq \max \left\lbrace \rho, \Lambda r, r^{1-\ell} \rho \right\rbrace V_m 
        = \Lambda r V_m, 
      \end{equation*}
      where the last equality hold by the choice on $\rho$. This concludes the first case of arbitrary transitions between coarse spaces.
      \item $m \in \C_i^k \setminus \I_{\C_i}^{k}, m+1 \in \F^{(\ell)}.$ Then
      \begin{align*}
      V_{m+1}  \leq b \norm{\grad f_{m+1}} \leq b M_{\C^{(\ell)}} \norm{\grad f_{m}}  
        \overset{k \in \C_i^k}{\leq} r^{\ell - i + 1} \frac{ M_{\C^{(\ell)}}}{\sigma} \norm{\grad f_{m}} \leq \Lambda r V_m.
      \end{align*}
      \item $m \in \F^{(\ell), k} \setminus \I_{\F^{(\ell)} }^{k} (\rho), m+1 \in \C_i$. Then, 
      \begin{align*}
      V_{m+1} = a_i \norm{R_i \grad f_{m+1}} \leq a_i \bar{\omega} \rho \norm{\grad f_{m}} = \frac{\bar{\omega} \rho}{r^{\ell-i+1}} V_m \leq \frac{\bar{\omega} \rho}{r^{\ell}} V_m.
      \end{align*}
      \item $m \in \F^{(\ell), k} \setminus \I_{\F^{(\ell)} }^{k} (\rho), m+1 \in \F^{(\ell)}$, then we directly obtain 
              \begin{equation*}
          V_{m+1} \leq \rho V_m.
        \end{equation*}
    \end{enumerate}
    Therefore, by the definition of $\rho$,
    \begin{equation} \label{ineq:V-non-I}
      V_{m+1} \leq \left\lbrace \rho, \Lambda r, \frac{\bar{\omega} \rho}{r^{\ell}} \right\rbrace V_m 
      = \Lambda r V_m < V_m, \qquad m \in \big(\I_\ell^k(\rho)\big)^{\rm c},
    \end{equation}
    which shows that $V_k$ is a contraction for elements of  $(\I_\ell^k(\rho))^{\rm c}$. 
    
    Assume now that $m \in \I_\ell^k(\rho)$. Similarly, we split our arguments into four cases, to obtain a recurrent relation for $V_k$. However, since $\I_\ell^k(\rho)$ contains a few elements only, $V_k$ is not required to be a contraction.
    \begin{enumerate}
      \item $m \in C_i^k, m+1 \in C_j$. Then, 
      \begin{equation*}
      V_{m+1} = a_j \norm{R_j \grad f_{m+1}} \leq a_j \bar{\omega}  M_{\C^{(\ell)}} \norm{\grad f_{m}} 
      \leq \Lambda r^{j-i} V_m \leq r^{-\ell} V_m, 
      \end{equation*}
      where the last inequality uses the monotonicity of $x \mapsto r^x$, $r \in (0,1)$.
      \item $m \in C_i^k, m+1 \in  \F^{(\ell)}$. Then, 
      \begin{align*}
        V_{m+1} \leq r^{\ell} M_{\C^{(\ell)}} \norm{\grad f_m} \leq r^{\ell - i + 1} \Lambda V_m \leq V_m.
      \end{align*}
      \item $m, \in  \F^{(\ell), k}, m+1 \in \F^{(\ell)}$. Then, by \eqref{ineq:full-grad-bound},
      \begin{equation*}
        V_{m+1} \leq 2 V_m.
      \end{equation*}
      \item $m \in \F^{(\ell), k}, m+1 \in \C_i$. Then,
      \begin{equation*}
        V_{m+1} \leq a_i \norm{R_i \grad f_{m+1}} \leq 2 a_i \bar{\omega} \norm{\grad f_{m}} \leq 2 r^{-\ell} \bar{\omega} V_m.
      \end{equation*}
    \end{enumerate}
    As a result, the above four cases imply
    \begin{equation} \label{ineq:V-in-I}
      V_{m+1} \leq 2 \bar{\omega} r^{-\ell} V_m, \qquad m \in \I_\ell^k(\rho).
    \end{equation}
    Let us denote 
    \begin{equation*}
      \eta \coloneqq r^{1 - (1+\ell) \theta} \Lambda^{1-\theta} (2\bar{\omega})^\theta.
    \end{equation*}
    By the choices of $\theta$ and $\rho$ it easy to see that $r^{1 - (1+\ell) \theta} < 1$ and since $r < B$ we have $\eta < 1$. Hence, by \eqref{ineq:V-non-I} and \eqref{ineq:V-in-I}, we obtain 
    \begin{align*}
      V_k \leq \left(\Lambda r \right)^{k - \left|\I_\ell^k(\rho)\right|} \left(2 r^{-\ell} \bar{\omega}\right)^{\left|\I_\ell^k(\rho)\right|} V_0 \overset{\left|\I_\ell^k(\rho)\right| \le \theta k}{\leq}  \eta^k V_0,
    \end{align*}
    which shows that $V_k$ is a contraction. Now the desired result follows from the last inequality, convexity, and \eqref{ineq:grad-V-relation} as follows. Let 
    \[
      A_\ell= \max\left\{\frac{1}{\sigma r^{\ell-1}},\frac1{r^\ell} \right\}.
    \]
    Then
    \begin{align*}
      f(x_k) - f^* \leq D \norm{\grad f_k} \leq  D A_\ell V_k \leq  D A_\ell V_0 \eta^k
    \end{align*}
    which shows exponential rate for the case $\big|\I_\ell^k(\rho)\big| \leq \theta k$. Hence the two cases imply
    \[
      f(x_k)-f^* \leq \max\left\{ D A_\ell V_0\eta^k,\, \frac{1}{C_\ell^2(1+\theta k/2)^2}
      \right\},
    \] 
which concludes the proof.
\end{proof}

\Cref{thm:nested-operators} establishes a direct $\Op\big(k^{-2}\big)$ convergence bound for the multilevel method with row-nested restriction operators, without imposing any condition on the frequency of fine-level iterations. Although the asymptotic rate also follows for general restriction operators from the eventual fine-only behavior, the present analysis controls the convergence rate throughout the entire sequence, including coarse-dominated phases. In particular, the resulting bound does not depend on the iteration at which the method enters the fine-only regime. 

The main difficulty for achieving this result is that a coarse correction at level $i$ controls $\|R_i\nabla f_{k+1}\|$, but not necessarily the reduced gradient at the subsequently selected level. Row-nestedness resolves this mismatch through monotonicity of the reduced-gradient norms, which, together with level-dependent weights and uniform full-gradient stability, allows us to control arbitrary transitions between levels. The resulting weighted gradient measure contracts outside a set of non-\(\rho\)-contracting iterations, while its growth remains uniformly bounded otherwise. When non-\(\rho\)-contracting iterations are sufficiently frequent, the function-value decrease estimates instead yield the $\Op\big(k^{-2}\big)$ rate. Thus, the analysis accommodates arbitrarily long sequences of coarse corrections without requiring a prescribed frequency of fine iterations or a decrease of the full gradient at every coarse step.

\section{Experiments} \label{sec:experiments}

\begin{figure}[htbp]
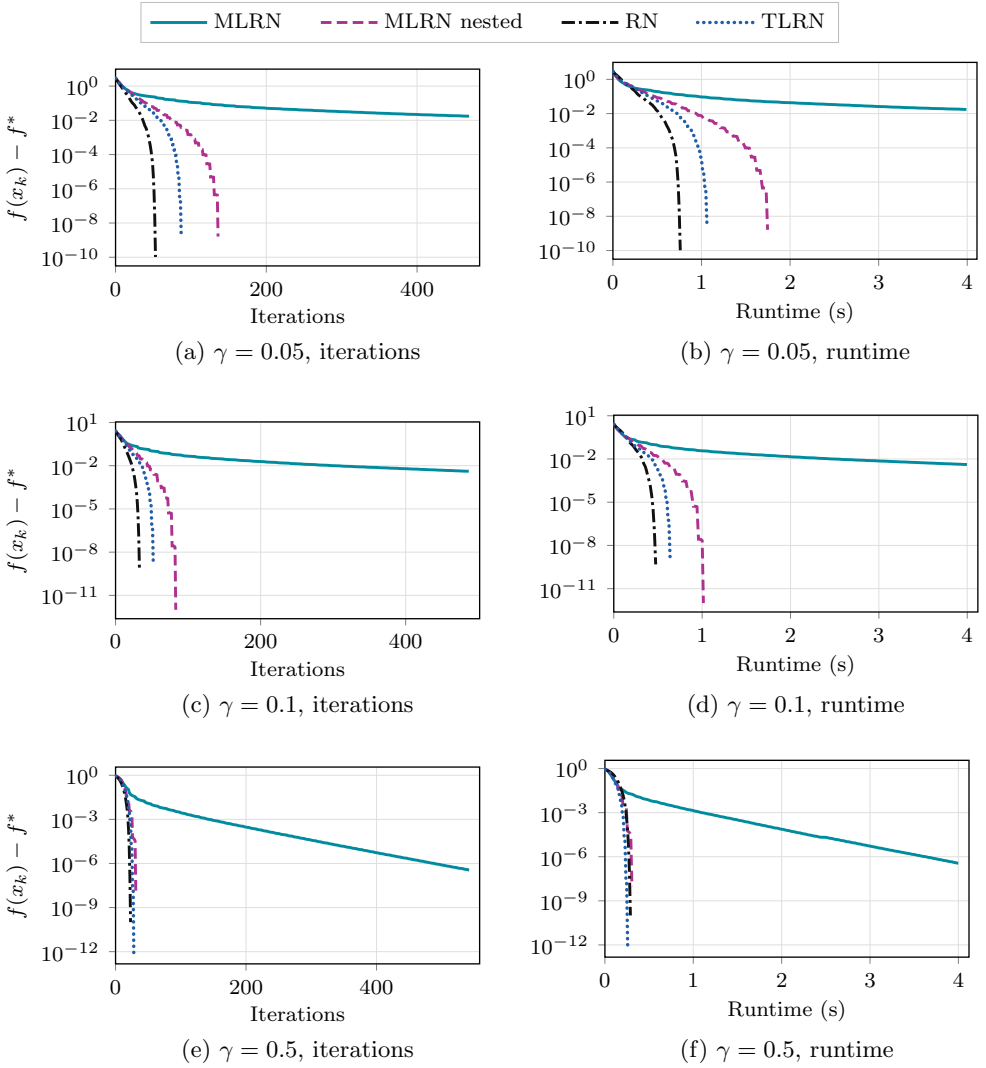

    \centering
    \ref{leg:logsum}\par\medskip

    \caption{Convergence comparison of MLRN, row-nested MLRN, regularized Newton (RN), and two-level
        regularized Newton (TLRN) for the log-sum-exp objective with $\gamma\in\{0.05,0.1,0.5\}$. The left
        column shows the objective gap $f(x_k)-f^*$ versus iterations, while the right column shows the
        objective gap versus wall-clock time. Row-nested MLRN and TLRN methods demonstrate comparable convergence behavior to RN.}
    \label{fig:logsum-comparison}
\end{figure}

In this section, we present numerical experiments to support our theoretical findings. In particular, we compare the performance of the two-level regularized Newton (TLRN), multilevel regularized Newton (MLRN), and row-nested multilevel regularized Newton (MLRN nested) methods with that of the full-space gradient-regularized Newton method (RN) \cite{mishchenko2023regularized}. Although our methods retain the $\mathcal{O}\big(k^{-2}\big)$ convergence rate of RN, the constants in their complexity bounds depend on $\sigma<1$, see \Cref{lemma:case-II}, potentially leading to less favorable theoretical bounds. Nevertheless, since the multilevel methods compute some of their updates in lower-dimensional spaces, they may offer significant computational savings. Our experiments aim to illustrate that the multilevel methods exhibit convergence behavior comparable to RN while potentially reducing the overall computational time.

The restriction operators for MLRN are constructed by selecting rows of the identity matrix as in \cite{Tsipinakis2025Adaptive}, whereas for the row-nested method we use the construction described in \Cref{rem:nested-orthonormal}. We consider $\ell=5$ coarse levels with dimensions $n_1=0.1n$, $n_2=0.3n$, $n_3=0.5n$, $n_4=0.7n$, and $n_5=0.9n$, whereas for TLRN we select $d=0.5n$. 

Since the Hessian's Lipschitz constant $L$ is generally unknown, we employ an adaptive line-search procedure, following the approach of \cite[Section 4.1.4.2]{nesterov2018lectures} and \cite{mishchenko2023regularized}. At each iteration, the estimate $L_k$ is increased until the corresponding sufficient-decrease and gradient-stability conditions are satisfied. We set $\varepsilon=10^{-10}$ and terminate the algorithms when $f(x_k)-f^* \leq \varepsilon_f$ or $\norm{\grad f(x_k)}\leq\varepsilon_g$, where $\varepsilon_f=10^{-8}$ and $\varepsilon_g=\varepsilon$. We use L-BFGS to compute a numerical approximation of the optimal objective value $f^*$. We set \(\sigma=0.5\), balancing the frequency of coarse corrections against their effectiveness (note that larger values of $\sigma$ yield more favorable convergence constants but impose a stricter coarse-level eligibility condition, potentially resulting in more fine-level iterations). All experiments are initialized at the origin.

\subsection{Log-sum-exp}
Given data $a_1,\ldots,a_m\in\R^n$ and $b_1,\ldots,b_m\in\R$,
we consider minimizing the log-sum-exp function
\begin{equation*}
    f_\gamma(x)
    = \gamma \ln\left(
        \sum_{i=1}^{m}
        \exp\left(\frac{a_i^\top x-b_i}{\gamma}\right)
    \right),
    \qquad \gamma>0.
    \label{eq:log-sum-exp}
\end{equation*}
The parameter $\gamma$ controls the smoothing of the log-sum-exp objective. We generate a random dataset with $m=1000$ observations and $n=800$ variables and consider three values of this parameter, namely $\gamma\in\{0.05,0.1,0.5\}$, to investigate the sensitivity of the algorithms to different degrees of smoothing.

The results are presented in \Cref{fig:logsum-comparison}. For all three values of $\gamma$, row-nested MLRN and TLRN exhibit convergence behavior comparable to that of RN, both in terms of iterations and computational time. In contrast, the non-nested MLRN method converges substantially more slowly. We suspect that this difference may be related to the fine-iteration frequency condition in \Cref{thm:multilevel-conditional}, which need not be satisfied during the observed iterations. Nevertheless, these findings suggest that row-nested restriction operators can be an attractive choice from both theoretical and computational perspectives.

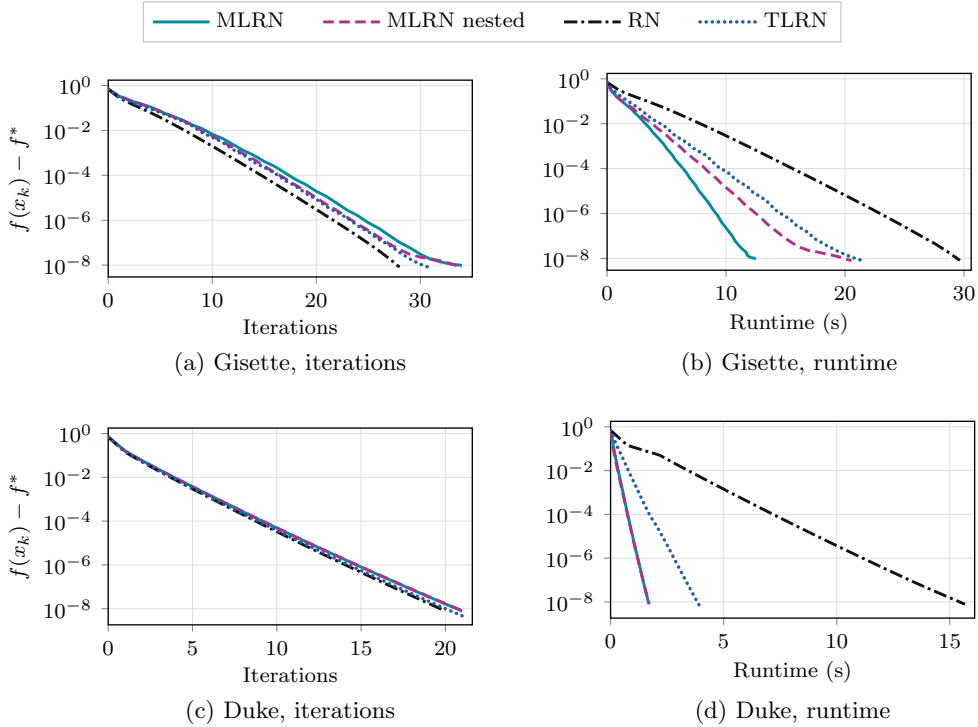
\begin{figure}[htbp]
    \centering
    \ref{leg:gisette-duke}\par\medskip
    \begin{tikzpicture}
    \begin{axis}[name=ax,
        mlrnaxis, ymode=log,
        xmin=0, xmax=35,
        ymin=2.968e-09, ymax=1.735e+00,
        ytickten={-8,-6,-4,-2,0},
        xlabel={Iterations},
        ylabel={$f(x_k)-f^*$},
        legend to name=leg:gisette-duke]
    \addplot[mlrn] table[row sep=\\] {
      x y \\
      0 0.693147 \\ 1 0.341234 \\ 2 0.221978 \\ 3 0.156725 \\ 4 0.114743 \\ 5 0.0712025 \\
      6 0.0448428 \\ 7 0.0312058 \\ 8 0.0182202 \\ 9 0.0121134 \\ 10 0.00679738 \\ 11 0.00437192 \\
      12 0.00236882 \\ 13 0.00128142 \\ 14 0.000787644 \\ 15 0.000413953 \\ 16 0.000249913 \\ 17 0.000128266 \\
      18 6.57309e-05 \\ 19 3.86296e-05 \\ 20 1.94499e-05 \\ 21 1.13001e-05 \\ 22 5.6e-06 \\ 23 2.77471e-06 \\
      24 1.58442e-06 \\ 25 7.74963e-07 \\ 26 4.38665e-07 \\ 27 2.12032e-07 \\ 28 1.03326e-07 \\ 29 5.91656e-08 \\
      30 3.07516e-08 \\ 31 1.83143e-08 \\ 32 1.36845e-08 \\ 33 1.0955e-08 \\ 34 9.77847e-09 \\
    };
    \addlegendentry{MLRN}
    \addplot[mlrnnested] table[row sep=\\] {
      x y \\
      0 0.693147 \\ 1 0.351924 \\ 2 0.22864 \\ 3 0.160374 \\ 4 0.106423 \\ 5 0.071223 \\
      6 0.0469322 \\ 7 0.0299318 \\ 8 0.0184503 \\ 9 0.00978718 \\ 10 0.00582648 \\ 11 0.00304582 \\
      12 0.00176339 \\ 13 0.000906382 \\ 14 0.000514131 \\ 15 0.000260414 \\ 16 0.000131736 \\ 17 7.33761e-05 \\
      18 3.6682e-05 \\ 19 2.01809e-05 \\ 20 9.99131e-06 \\ 21 4.94788e-06 \\ 22 2.69538e-06 \\ 23 1.32472e-06 \\
      24 6.52368e-07 \\ 25 3.54835e-07 \\ 26 1.76152e-07 \\ 27 9.00848e-08 \\ 28 4.92358e-08 \\ 29 3.04536e-08 \\
      30 2.23469e-08 \\ 31 1.91902e-08 \\ 32 1.51284e-08 \\ 33 1.13527e-08 \\ 34 7.85223e-09 \\
    };
    \addlegendentry{MLRN nested}
    \addplot[rn] table[row sep=\\] {
      x y \\
      0 0.693147 \\ 1 0.291616 \\ 2 0.169193 \\ 3 0.104108 \\ 4 0.0639569 \\ 5 0.0379529 \\
      6 0.0217411 \\ 7 0.0122087 \\ 8 0.00675044 \\ 9 0.00368603 \\ 10 0.00199245 \\ 11 0.00106787 \\
      12 0.000568102 \\ 13 0.000300249 \\ 14 0.000157754 \\ 15 8.24476e-05 \\ 16 4.2884e-05 \\ 17 2.22081e-05 \\
      18 1.14541e-05 \\ 19 5.88436e-06 \\ 20 3.01082e-06 \\ 21 1.53353e-06 \\ 22 7.76618e-07 \\ 23 3.90147e-07 \\
      24 1.93615e-07 \\ 25 9.42246e-08 \\ 26 4.44096e-08 \\ 27 1.98449e-08 \\ 28 8.10823e-09 \\
    };
    \addlegendentry{RN}
    \addplot[tlrn] table[row sep=\\] {
      x y \\
      0 0.693147 \\ 1 0.323456 \\ 2 0.200904 \\ 3 0.134309 \\ 4 0.0921445 \\ 5 0.0629467 \\
      6 0.0420071 \\ 7 0.0270752 \\ 8 0.0168813 \\ 9 0.00809604 \\ 10 0.00492521 \\ 11 0.00292718 \\
      12 0.00139256 \\ 13 0.000815136 \\ 14 0.000385658 \\ 15 0.000223043 \\ 16 0.000126941 \\ 17 5.89582e-05 \\
      18 3.33341e-05 \\ 19 1.54134e-05 \\ 20 8.66531e-06 \\ 21 3.98938e-06 \\ 22 2.23132e-06 \\ 23 1.23211e-06 \\
      24 5.58869e-07 \\ 25 3.0866e-07 \\ 26 1.3972e-07 \\ 27 7.88348e-08 \\ 28 3.70373e-08 \\ 29 1.88081e-08 \\
      30 1.09698e-08 \\ 31 7.42987e-09 \\
    };
    \addlegendentry{TLRN}
    \end{axis}
    \node[anchor=north,font=\small] at (ax.below south) {(a) Gisette, iterations};
    \end{tikzpicture}
    \hfill
    \begin{tikzpicture}
    \begin{axis}[name=ax,
        mlrnaxis, ymode=log,
        xmin=0, xmax=30.6,
        ymin=2.968e-09, ymax=1.735e+00,
        ytickten={-8,-6,-4,-2,0},
        xlabel={Runtime (s)}]
    \addplot[mlrn] table[row sep=\\] {
      x y \\
      0 0.693147 \\ 0.289378 0.341234 \\ 0.557436 0.221978 \\ 0.821544 0.156725 \\ 1.08214 0.114743 \\ 1.58452 0.0712025 \\
      2.04477 0.0448428 \\ 2.30896 0.0312058 \\ 2.77607 0.0182202 \\ 3.07122 0.0121134 \\ 3.53687 0.00679738 \\ 3.80257 0.00437192 \\
      4.26728 0.00236882 \\ 4.7345 0.00128142 \\ 5.00719 0.000787644 \\ 5.46636 0.000413953 \\ 5.73065 0.000249913 \\ 6.21407 0.000128266 \\
      6.67125 6.57309e-05 \\ 6.93631 3.86296e-05 \\ 7.39711 1.94499e-05 \\ 7.66253 1.13001e-05 \\ 8.12234 5.6e-06 \\ 8.58261 2.77471e-06 \\
      8.86199 1.58442e-06 \\ 9.31813 7.74963e-07 \\ 9.58241 4.38665e-07 \\ 10.042 2.12032e-07 \\ 10.5024 1.03326e-07 \\ 10.7664 5.91656e-08 \\
      11.2271 3.07516e-08 \\ 11.6939 1.83143e-08 \\ 11.7873 1.36845e-08 \\ 12.0539 1.0955e-08 \\ 12.5185 9.77847e-09 \\
    };
    \addplot[mlrnnested] table[row sep=\\] {
      x y \\
      0 0.693147 \\ 0.267127 0.351924 \\ 0.529526 0.22864 \\ 0.790137 0.160374 \\ 1.24833 0.106423 \\ 1.71108 0.071223 \\
      2.17185 0.0469322 \\ 2.63325 0.0299318 \\ 3.09394 0.0184503 \\ 3.82027 0.00978718 \\ 4.3024 0.00582648 \\ 5.00095 0.00304582 \\
      5.46795 0.00176339 \\ 6.16167 0.000906382 \\ 6.62809 0.000514131 \\ 7.32517 0.000260414 \\ 8.02063 0.000131736 \\ 8.48979 7.33761e-05 \\
      9.18555 3.6682e-05 \\ 9.65402 2.01809e-05 \\ 10.3486 9.99131e-06 \\ 11.0309 4.94788e-06 \\ 11.5176 2.69538e-06 \\ 12.234 1.32472e-06 \\
      12.9144 6.52368e-07 \\ 13.4224 3.54835e-07 \\ 14.1379 1.76152e-07 \\ 14.8347 9.00848e-08 \\ 15.5462 4.92358e-08 \\ 16.2597 3.04536e-08 \\
      16.946 2.23469e-08 \\ 17.6668 1.91902e-08 \\ 18.6288 1.51284e-08 \\ 19.5943 1.13527e-08 \\ 20.5477 7.85223e-09 \\
    };
    \addplot[rn] table[row sep=\\] {
      x y \\
      0 0.693147 \\ 1.08027 0.291616 \\ 2.21109 0.169193 \\ 3.29773 0.104108 \\ 4.33974 0.0639569 \\ 5.38297 0.0379529 \\
      6.426 0.0217411 \\ 7.47015 0.0122087 \\ 8.5293 0.00675044 \\ 9.5925 0.00368603 \\ 10.6371 0.00199245 \\ 11.6978 0.00106787 \\
      12.7447 0.000568102 \\ 13.7868 0.000300249 \\ 14.8456 0.000157754 \\ 15.9064 8.24476e-05 \\ 16.9845 4.2884e-05 \\ 18.0453 2.22081e-05 \\
      19.1044 1.14541e-05 \\ 20.1633 5.88436e-06 \\ 21.223 3.01082e-06 \\ 22.287 1.53353e-06 \\ 23.3288 7.76618e-07 \\ 24.3738 3.90147e-07 \\
      25.418 1.93615e-07 \\ 26.4628 9.42246e-08 \\ 27.5762 4.44096e-08 \\ 28.6168 1.98449e-08 \\ 29.6604 8.10823e-09 \\
    };
    \addplot[tlrn] table[row sep=\\] {
      x y \\
      0 0.693147 \\ 0.472023 0.323456 \\ 0.940126 0.200904 \\ 1.40833 0.134309 \\ 1.87575 0.0921445 \\ 2.34369 0.0629467 \\
      2.81126 0.0420071 \\ 3.27368 0.0270752 \\ 3.73204 0.0168813 \\ 4.78667 0.00809604 \\ 5.25437 0.00492521 \\ 5.71603 0.00292718 \\
      6.76159 0.00139256 \\ 7.23062 0.000815136 \\ 8.32499 0.000385658 \\ 8.79135 0.000223043 \\ 9.25218 0.000126941 \\ 10.299 5.89582e-05 \\
      10.7667 3.33341e-05 \\ 11.8089 1.54134e-05 \\ 12.2749 8.66531e-06 \\ 13.3293 3.98938e-06 \\ 13.805 2.23132e-06 \\ 14.3121 1.23211e-06 \\
      15.3669 5.58869e-07 \\ 15.864 3.0866e-07 \\ 16.9224 1.3972e-07 \\ 17.389 7.88348e-08 \\ 18.4801 3.70373e-08 \\ 19.5553 1.88081e-08 \\
      20.6 1.09698e-08 \\ 21.6455 7.42987e-09 \\
    };
    \end{axis}
    \node[anchor=north,font=\small] at (ax.below south) {(b) Gisette, runtime};
    \end{tikzpicture}

    \vspace{0.8em}

    \begin{tikzpicture}
    \begin{axis}[name=ax,
        mlrnaxis, ymode=log,
        xmin=0, xmax=21.6,
        ymin=1.826e-09, ymax=1.776e+00,
        ytickten={-8,-6,-4,-2,0},
        xlabel={Iterations},
        ylabel={$f(x_k)-f^*$}]
    \addplot[mlrn] table[row sep=\\] {
      x y \\
      0 0.693147 \\ 1 0.15956 \\ 2 0.0581917 \\ 3 0.0227258 \\ 4 0.00912504 \\ 5 0.00372137 \\
      6 0.0015348 \\ 7 0.000639153 \\ 8 0.000268621 \\ 9 0.000113924 \\ 10 4.87567e-05 \\ 11 2.10567e-05 \\
      12 9.17541e-06 \\ 13 4.03382e-06 \\ 14 1.78946e-06 \\ 15 8.01012e-07 \\ 16 3.61722e-07 \\ 17 1.64745e-07 \\
      18 7.56655e-08 \\ 19 3.50577e-08 \\ 20 1.64112e-08 \\ 21 7.7959e-09 \\
    };
    \addplot[mlrnnested] table[row sep=\\] {
      x y \\
      0 0.693147 \\ 1 0.159741 \\ 2 0.0582728 \\ 3 0.0227713 \\ 4 0.00915124 \\ 5 0.0037363 \\
      6 0.00154315 \\ 7 0.000643727 \\ 8 0.000271074 \\ 9 0.000115212 \\ 10 4.94217e-05 \\ 11 2.13962e-05 \\
      12 9.34789e-06 \\ 13 4.12087e-06 \\ 14 1.83279e-06 \\ 15 8.22337e-07 \\ 16 3.72166e-07 \\ 17 1.69853e-07 \\
      18 7.81665e-08 \\ 19 3.6286e-08 \\ 20 1.70185e-08 \\ 21 8.10038e-09 \\
    };
    \addplot[rn] table[row sep=\\] {
      x y \\
      0 0.693147 \\ 1 0.141749 \\ 2 0.0504893 \\ 3 0.0192 \\ 4 0.00749097 \\ 5 0.00296374 \\
      6 0.00118443 \\ 7 0.000477524 \\ 8 0.000194157 \\ 9 7.96153e-05 \\ 10 3.29305e-05 \\ 11 1.3742e-05 \\
      12 5.78676e-06 \\ 13 2.45928e-06 \\ 14 1.05477e-06 \\ 15 4.56483e-07 \\ 16 1.99339e-07 \\ 17 8.78302e-08 \\
      18 3.90267e-08 \\ 19 1.74686e-08 \\ 20 7.86086e-09 \\
    };
    \addplot[tlrn] table[row sep=\\] {
      x y \\
      0 0.693147 \\ 1 0.150709 \\ 2 0.05382 \\ 3 0.0205865 \\ 4 0.0080899 \\ 5 0.00322643 \\
      6 0.00130055 \\ 7 0.000529112 \\ 8 0.00021718 \\ 9 8.99358e-05 \\ 10 3.75784e-05 \\ 11 1.58451e-05 \\
      12 6.74293e-06 \\ 13 2.8961e-06 \\ 14 1.25535e-06 \\ 15 5.49125e-07 \\ 16 2.42413e-07 \\ 17 1.07997e-07 \\
      18 4.85483e-08 \\ 19 2.20212e-08 \\ 20 1.00858e-08 \\ 21 4.67671e-09 \\
    };
    \end{axis}
    \node[anchor=north,font=\small] at (ax.below south) {(c) Duke, iterations};
    \end{tikzpicture}
    \hfill
    \begin{tikzpicture}
    \begin{axis}[name=ax,
        mlrnaxis, ymode=log,
        xmin=0, xmax=16.1,
        ymin=1.826e-09, ymax=1.776e+00,
        ytickten={-8,-6,-4,-2,0},
        xlabel={Runtime (s)}]
    \addplot[mlrn] table[row sep=\\] {
      x y \\
      0 0.693147 \\ 0.107316 0.15956 \\ 0.189849 0.0581917 \\ 0.269627 0.0227258 \\ 0.348134 0.00912504 \\ 0.426704 0.00372137 \\
      0.510825 0.0015348 \\ 0.587515 0.000639153 \\ 0.664879 0.000268621 \\ 0.7443 0.000113924 \\ 0.823644 4.87567e-05 \\ 0.902933 2.10567e-05 \\
      0.984368 9.17541e-06 \\ 1.0645 4.03382e-06 \\ 1.14271 1.78946e-06 \\ 1.22532 8.01012e-07 \\ 1.30544 3.61722e-07 \\ 1.382 1.64745e-07 \\
      1.4615 7.56655e-08 \\ 1.53983 3.50577e-08 \\ 1.61899 1.64112e-08 \\ 1.69816 7.7959e-09 \\
    };
    \addplot[mlrnnested] table[row sep=\\] {
      x y \\
      0 0.693147 \\ 0.075281 0.159741 \\ 0.161511 0.0582728 \\ 0.241713 0.0227713 \\ 0.324653 0.00915124 \\ 0.407831 0.0037363 \\
      0.489157 0.00154315 \\ 0.572311 0.000643727 \\ 0.653658 0.000271074 \\ 0.732206 0.000115212 \\ 0.810311 4.94217e-05 \\ 0.895539 2.13962e-05 \\
      0.974757 9.34789e-06 \\ 1.05202 4.12087e-06 \\ 1.13764 1.83279e-06 \\ 1.21614 8.22337e-07 \\ 1.29337 3.72166e-07 \\ 1.37782 1.69853e-07 \\
      1.45812 7.81665e-08 \\ 1.54099 3.6286e-08 \\ 1.62539 1.70185e-08 \\ 1.707 8.10038e-09 \\
    };
    \addplot[rn] table[row sep=\\] {
      x y \\
      0 0.693147 \\ 0.741092 0.141749 \\ 2.16882 0.0504893 \\ 2.9132 0.0192 \\ 3.66182 0.00749097 \\ 4.40453 0.00296374 \\
      5.14713 0.00118443 \\ 5.89295 0.000477524 \\ 6.63739 0.000194157 \\ 7.38009 7.96153e-05 \\ 8.12174 3.29305e-05 \\ 8.86554 1.3742e-05 \\
      9.61025 5.78676e-06 \\ 10.3585 2.45928e-06 \\ 11.1026 1.05477e-06 \\ 11.8472 4.56483e-07 \\ 12.5905 1.99339e-07 \\ 13.3354 8.78302e-08 \\
      14.1053 3.90267e-08 \\ 14.8898 1.74686e-08 \\ 15.6721 7.86086e-09 \\
    };
    \addplot[tlrn] table[row sep=\\] {
      x y \\
      0 0.693147 \\ 0.259394 0.150709 \\ 0.463393 0.05382 \\ 0.651306 0.0205865 \\ 0.830558 0.0080899 \\ 1.01122 0.00322643 \\
      1.19159 0.00130055 \\ 1.36947 0.000529112 \\ 1.55353 0.00021718 \\ 1.73696 8.99358e-05 \\ 1.96636 3.75784e-05 \\ 2.18405 1.58451e-05 \\
      2.38181 6.74293e-06 \\ 2.57252 2.8961e-06 \\ 2.75113 1.25535e-06 \\ 2.93384 5.49125e-07 \\ 3.11331 2.42413e-07 \\ 3.29476 1.07997e-07 \\
      3.47367 4.85483e-08 \\ 3.65467 2.20212e-08 \\ 3.83513 1.00858e-08 \\ 4.01907 4.67671e-09 \\
    };
    \end{axis}
    \node[anchor=north,font=\small] at (ax.below south) {(d) Duke, runtime};
    \end{tikzpicture}
    \caption{Convergence comparison of MLRN, row-nested MLRN, regularized Newton (RN), and two-level regularized Newton (TLRN) using the large Gisette and Duke datasets for logistic regression. The left column reports the objective gap $f(x_k)-f^*$ versus iterations, while the right column reports the objective gap versus wall-clock time. The multilevel gradient-regularized Newton methods offer substantial computational savings compared to their full-space counterpart.}
    \label{fig:logistic-gisette-duke}
\end{figure}

\begin{figure}[htbp]
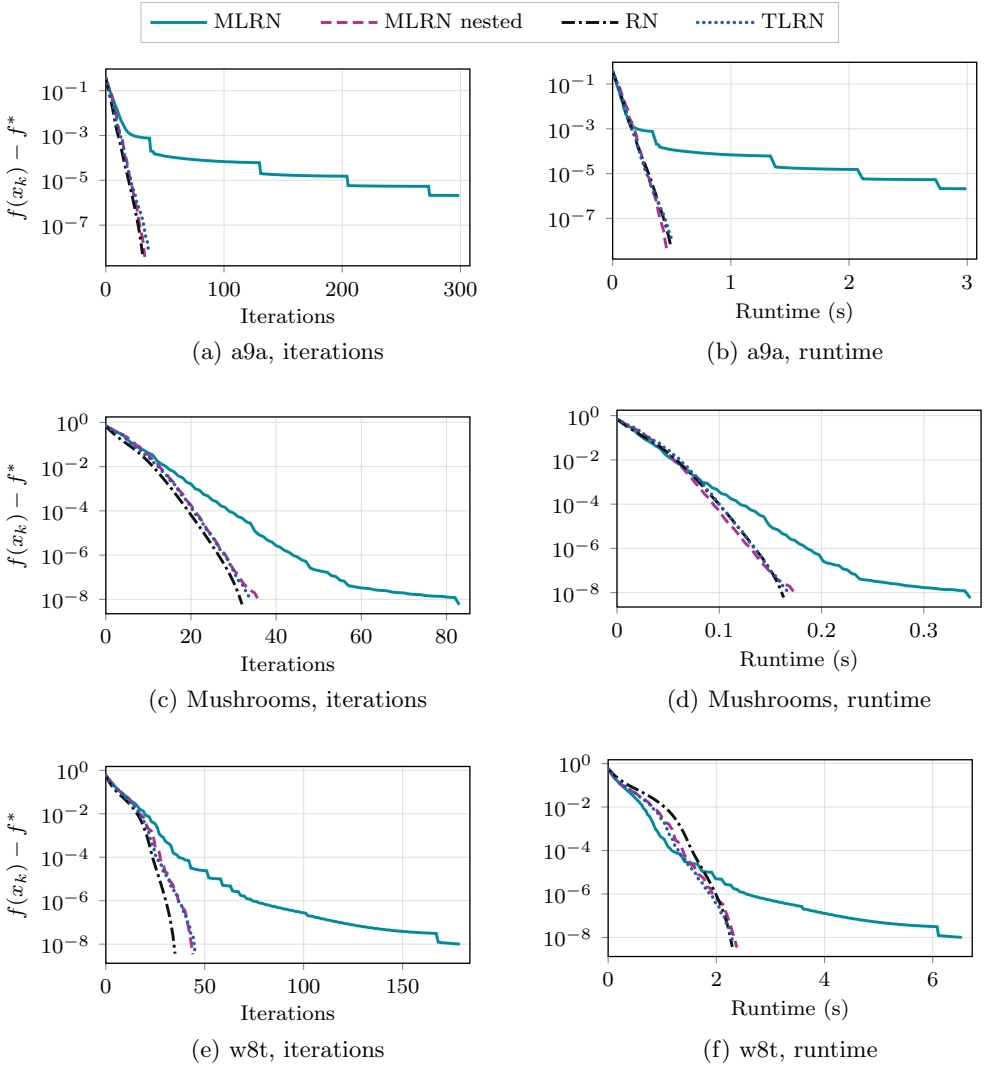

    \centering
    \ref{leg:a9a-mushrooms-w8t}\par\medskip

    \caption{Convergence comparison of MLRN, row-nested MLRN, regularized Newton (RN), and two-level regularized Newton (TLRN) on the a9a, Mushrooms, and w8t datasets for logistic regression. The left column reports the objective gap $f(x_k)-f^*$ versus iterations, while the right column reports the objective gap versus wall-clock time. Row-nested restriction operators significantly improve the convergence behavior of MLRN.}
    \label{fig:logistic-a9a-mushrooms-w8t}
\end{figure}

\subsection{Logistic regression}
In the second experiment, we consider logistic regression problems using real-world datasets. The
objective function is given by
\begin{equation}
    f(x)=\frac{1}{m}\sum_{i=1}^{m}
    \ln\left(1+\exp\big(-b_i a_i^\top x\big)\right)
    +\frac{\lambda_2}{2}\norm{x}^2,
    \label{eq:logistic-regression}
\end{equation}
where $b_i\in\{-1,1\}$ and we set $\lambda_2=10^{-10}$ to make the problem ill-conditioned. The results are presented in \Cref{fig:logistic-gisette-duke} and \Cref{fig:logistic-a9a-mushrooms-w8t}.

Across the considered datasets, nested MLRN and TLRN exhibit convergence behavior comparable to that of the full-space RN method. Moreover, as shown in \Cref{fig:logistic-gisette-duke}, when applied to large datasets, both methods achieve substantial reductions in computational time, illustrating the benefits of performing corrections in lower-dimensional spaces. In contrast, non-row-nested MLRN exhibits slower convergence on some instances, particularly those shown in \Cref{fig:logistic-a9a-mushrooms-w8t}, which may prevent the method from converging faster than RN in computational time. These observations suggest that the structure of the restriction operators, particularly row-nestedness, can play an important role in the practical performance of multilevel gradient-regularized Newton methods. Overall, the results demonstrate that row-nested restriction operators can yield computationally efficient multilevel methods,
complementing the theoretical convergence guarantees established in this paper.

\section{Conclusions} \label{sec:conclusions}

We developed and analyzed multilevel gradient-regularized \sloppy Newton methods for unconstrained convex optimization, establishing a global $\Op\big(k^{-2}\big)$ convergence rate in function values. Among our results, the proof of \Cref{thm:nested-operators} is of particular interest as it combines a weighted gradient measure with row-nestedness, a special structure of the restriction operators introduced in \Cref{sec:nested-operators} for the purposes of this analysis. These ideas are new in multilevel optimization theory and we believe that they could prove useful for establishing fast convergence rates for other multilevel methods, such as cubic-regularized Newton and higher-order methods.

\section*{Code and data availability} 
The code for reproducing the numerical results is available at\medskip
\begin{center}
\url{https://github.com/Ntsip/Multilevel-Gradient-Regularized-Newton-Method}.
\end{center} \medskip
It is authored by Nick Tsipinakis and published under the MIT license. The benchmark data used in the experiments is taken from \medskip 
\begin{center}
\url{https://www.csie.ntu.edu.tw/~cjlin/libsvmtools/datasets/}.
\end{center}

\section*{CRediT author statement} \textbf{Nick Tsipinakis:} Conceptualization, Methodology, Software, Formal Analysis, Investigation, Writing --- Original Draft, Visualization; \textbf{Matthias Voigt:} Writing --- Review \& Editing, Visualization, Supervision.

\section*{Acknowledgments}
The authors used AI tools to assist with language and figure editing, mathematical proof verification, and the development and debugging of numerical implementations. All mathematical results,
algorithms, and computational experiments were reviewed and validated by the authors. The authors assume responsibility for all content.

\bibliographystyle{siamplain}
\bibliography{references}


\appendix

\section{Background and proofs of auxiliary results}
\label{app:proofs}

\begin{proposition}
\label{prop:recurrence-rate}
Let ${(\Delta_\ell)}_{\ell=0}^{\tau}$ be a tuple of nonnegative numbers satisfying
\[
\Delta_{\ell+1}
\leq
\Delta_\ell-C\Delta_\ell^{3/2},
\qquad \ell=0,\ldots,\tau-1,
\]
for some $C>0$ and $\tau \ge 1$. Then
\[
\Delta_\tau
\leq
\frac{1}{C^2(1+\tau/2)^2}.
\]
\end{proposition}

\begin{proof}
If \(\Delta_\ell=0\) for some \(\ell\), the recurrence and nonnegativity imply \(\Delta_j=0\) for all \(j\geq\ell\), so the result follows immediately. Otherwise,
nonnegativity implies $C\sqrt{\Delta_\ell} < 1$, and hence
\[
\frac{1}{\sqrt{\Delta_{\ell+1}}}
\geq
\frac{1}{\sqrt{\Delta_\ell}}
\frac{1}{\sqrt{1-C\sqrt{\Delta_\ell}}}
\geq
\frac{1}{\sqrt{\Delta_\ell}}+\frac{C}{2},
\]
where we used $(1-s)^{-1/2}\geq1+s/2$ for $s\in[0,1)$. Summing $\frac{1}{\sqrt{\Delta_{\ell+1}}}- \frac{1}{\sqrt{\Delta_{\ell}}}$ over
$\ell=0,\ldots,\tau-1$, exploiting the telescoping property, and using
$1/\sqrt{\Delta_0}\geq C$ gives
\[
\frac{1}{\sqrt{\Delta_\tau}}
\geq
C\left(1+\frac{\tau}{2}\right),
\]
which proves the claim.
\end{proof}
We also refer to  \cite[Theorem 4.1.4]{nesterov2018lectures} for a related recurrence estimate.

\begin{lemma}\label{lemma:finite-coarse}
Suppose that Assumptions~\ref{ass:bounded-sublevel-sets}, \ref{ass:lipschitz-hessians}, and \ref{ass:R_i} hold.
Then \sloppy \Cref{alg:multilevel} performs only finitely many coarse iterations.
\end{lemma}

\begin{proof}
Let $k\in\mathcal C_i$ for some $i\in\mathcal J$, and define
\[
\alpha_{i,k}\coloneqq
\sqrt{\frac{\omega_i^3L\norm{R_i\grad f_k}}{2}},
\qquad
d_{i,k}\coloneqq
-\left(R_i\nabla^2 f_kP_i+\alpha_{i,k}I_{n_i}\right)^{-1}
R_i\grad f_k.
\]
By the coarse-level eligibility condition,
\[
\alpha_{i,k}\geq
\sqrt{\frac{\omega_i^3L\varepsilon}{2}}>0.
\]
Moreover, since the iterates remain in the compact set
$\mathcal L(x_0)$, there exists 
\[
G\coloneqq\max_{x\in\mathcal L(x_0)}\norm{\grad f(x)}<\infty.
\]
Consequently,
\[
\alpha_{i,k}\leq
\bar\alpha_i\coloneqq\omega_i^2\sqrt{\frac{LG}{2}}.
\]
Using the definition of $d_{i,k}$, we obtain
\[
\varepsilon
\leq \norm{R_i\grad f_k}
\leq
\left(\omega_i^2H+\alpha_{i,k}\right)
\norm{d_{i,k}}
\leq
\left(\omega_i^2H+\bar\alpha_i\right)
\norm{d_{i,k}}.
\]
Hence,
\[
\norm{d_{i,k}}\geq
\frac{\varepsilon}{\omega_i^2H+\bar\alpha_i}.
\]

By the descent estimate in
\cite[Lemma 3.3]{tsipinakis2025regularized},
the coarse update satisfies
\[
f(x_k)-f(x_{k+1})
\geq \frac{\alpha_{i,k}}{2}\norm{d_{i,k}}^2.
\]
Combining the preceding estimates yields
\[
f(x_k)-f(x_{k+1})
\geq
\underbrace{
\frac12\sqrt{\frac{\omega_i^3L\varepsilon}{2}}
\left(
\frac{\varepsilon}{\omega_i^2H+\bar\alpha_i}
\right)^2
}_{\eqqcolon \delta_i>0}.
\]
Since the number of coarse levels is finite,
$\delta\coloneqq\min_{i\in\mathcal J}\delta_i>0$.
Therefore, every coarse iteration decreases the objective
function by at least $\delta$. By monotonicity of the
function values and their lower boundedness by $f^*$,
\[
\big|\mathcal C^{(\ell)}\big| \delta
\leq f(x_0)-f^*.
\]
Thus, $\big|\mathcal C^{(\ell)}\big|<\infty$, proving the result.
\end{proof}

\subsection{Proof of \Cref{lemma:multilevel-coarse-estimates}}
\label{sec:multilevel-coarse-estimates}

\begin{proof}
The result follows from the analysis in
\cite[Lemma~3.2 and Corollary~3.4]{tsipinakis2025regularized},
which allows the restriction operator to vary with the iteration.
Indeed, at an iteration $k\in\C_i$, applying the corresponding
estimates with $R_k=R_i$ and $\omega_i=\|R_i\|$ yields
\[
\|R_i\nabla f_{k+1}\|
\leq
2\alpha_k\|d_{i,k}\|,
\qquad
\|d_{i,k}\|
\leq
\frac{\|R_i\nabla f_k\|}{\alpha_k},
\]
where
\[
\alpha_k=
\sqrt{\frac{\omega_i^3L\|R_i\nabla f_k\|}{2}}, \qquad 
d_{i,k}\coloneqq
-\left(R_i\nabla^2f_kP_i+\alpha_kI_{n_i}\right)^{-1}R_i\nabla f_k.
\]
These inequalities yield
\eqref{ineq:multilevel-reduced-gradient-stability}, while the proof of
\cite[Corollary~3.4]{tsipinakis2025regularized} gives
\eqref{ineq:multilevel-reduction}.
\end{proof}

\subsection{Proof of \Cref{lemma:case-II}}
\label{sec:case-II}
\begin{proof}
The proof of Case II in \Cref{thm:two-level-mixed} extends naturally to $\ell$ coarse levels. Thus, below, only the main steps of the proof will be discussed. Let $\tau\coloneqq\big|\mathcal I_\ell^k(\rho)\big|$ and enumerate
$\mathcal I_\ell^k(\rho)=\{k_0,\ldots,k_{\tau-1}\}$ increasingly,
with $k_\tau\coloneqq k$. If
$k_j\in\mathcal I_{\mathcal C_i}^k(\rho)$ for some $i\in\mathcal J$,
then, using the coarse decrease estimate $\omega_i\leq\zeta$ and
$\|R_i\nabla f_{k_j}\|\geq\sigma\|\nabla f_{k_j}\|$, we obtain
\[
f(x_{k_j})-f(x_{k_j+1})
\geq
\frac{\sqrt{2}\rho^2}{8\sqrt L}
\left(\frac{\sigma}{\zeta}\right)^{3/2}
\|\nabla f_{k_j}\|^{3/2}.
\]
The same bound holds for
$k_j\in\mathcal I_{\mathcal F^{(\ell)}}^k(\rho)$ by the fine-level
decrease estimate. Hence, arguing as in Case~II of
\Cref{thm:two-level-mixed},
\[
\Delta_{j+1}
\leq
\Delta_j-C_\ell\Delta_j^{3/2},
\qquad
\Delta_j\coloneqq f(x_{k_j})-f^*.
\]
Since $\tau>\theta k$, Proposition~\ref{prop:recurrence-rate} yields
\[
f(x_k)-f^*
\leq
\frac{1}{C_\ell^2(1+\theta k/2)^2}.
\]
\end{proof}


\end{document}